\documentclass[12pt]{amsart}
\usepackage{palatino}
 \usepackage[english]{babel}
 \usepackage[latin1]{inputenc}
 \usepackage{amsmath,amssymb,amsthm, fancyhdr,latexsym}
 \usepackage{eurosym,enumerate}
\usepackage{graphicx}
\usepackage[usenames, dvipsnames]{color}
 \usepackage{tikz}
\usetikzlibrary{arrows,shapes,matrix,decorations.pathmorphing}
\definecolor{MT0}{RGB}{219, 247, 255}
\definecolor{MT1}{RGB}{153, 230, 255}
\definecolor{MT2}{RGB}{30, 144, 255}
\definecolor{MT3}{RGB}{31,199,255}
\definecolor{MT4}{RGB}{0,51,255}
\definecolor{MT5}{RGB}{0,46,184}
\definecolor{MT6}{RGB}{0,0,128}
 \newtheorem{theorem}{Theorem}[section]
 \newtheorem{corollary}[theorem]{Corollary}
 \newtheorem{proposition}[theorem]{Proposition}
 \newtheorem{lemma}[theorem]{Lemma}

 \theoremstyle{definition}
 \newtheorem{definition}[theorem]{Definition}
 \newtheorem{example}[theorem]{Example}
 \newtheorem{remark}[theorem]{Remark}

 \makeindex
 \newcommand{\TT}{\mathcal{T}}
 \newcommand{\NN}{\mathbb{N}}
 \newcommand{\ZZ}{\mathbb{Z}}

 \newcommand{\Prec}{\pmb{<}}
 \newcommand{\Succ}{\pmb{>}}
 \newcommand{\PPrec}{\pmb{<\!\!<}}
 \newcommand{\SSucc}{\pmb{>\!\!>}}
\newcommand{\Preceq}{\pmb{\leq}}
\newcommand{\Succeq}{\pmb{\geq}}
\newcommand{\cT}{{\sf{T}}}
\newcommand{\cN}{{\sf{N}}}
 \newcommand{\cB}{{\sf{B}}}
 \newcommand{\Ht}{\textrm{Ht}}
 \newcommand{\supp}{\textrm{Supp}}
 \newcommand{\cone}{\textrm{cone}}

 \newcommand{\LR}{\textrm{LRep}}
 \newcommand{\SLR}{\textrm{SLRep}}
 \newcommand{\PP}{\mathcal{P}}
  \newcommand{\lcm}{\textrm{lcm}}
\def\Either{\mbox{\ either }}

\def\Or{\mbox{\ or }}

\definecolor{bcP}{rgb}{1,0,1}

\newcommand{\MFFunctor}[1]{\underline{\mathbf{Mf}}_{#1}}
\newcommand{\MFScheme}[1]{\mathbf{Mf}_{#1}}
\newcommand{\Sets}{\underline{\textnormal{Set}}}
 
\newcommand{\MSFunctor}[1]{\underline{\mathbf{Ms}}_{#1}}

\usepackage[pdftex,bookmarks,colorlinks]{hyperref} 

 \title[A general framework for Noetherian well ordered polynomial reductions]{A general framework for Noetherian well ordered polynomial reductions}

\author[M.Ceria]{Michela Ceria}
\address{Department of Computer Science of University of Milan, Via Comelico 39, Milano, Italy}
\email{\href{mailto:michela.ceria@gmail.com}{michela.ceria@gmail.com}}

\author[T.Mora]{Teo Mora}
\address{Dipartimento di Matematica dell'Universit\`{a} di Genova,         Via Dodecaneso 35,
  16146  Genova,  Italy.
         }
\email{\href{mailto:theomora@disi.unige.it}{theomora@disi.unige.it}}

\author[M.Roggero]{Margherita Roggero}
\address{Dipartimento di Matematica dell'Universit\`{a} di Torino\\ 
         Via Carlo Alberto 10, 
         10123 Torino, Italy}
\email{\href{mailto:margherit
a.roggero@unito.it}{margherita.roggero@unito.it}}

\keywords{Polynomial reduction, ideal membership, well-founded order}
\subjclass[2000]{14C05, 14Q20, 13P10}
\begin{document}
\newpage
\fussy

\begin{abstract}
Polynomial  reduction is one of the main tools in computational algebra with
innumerable applications in many areas, both pure and applied. Since many years
both the theory and an efficient  design of the related algorithm
have been solidly established.
This paper presents a general definition of polynomial reduction structure,
studies its features and highlights the
aspects needed in order to grant and to efficiently test the main properties
(noetherianity, confluence, ideal membership).
The most significant aspect of this analysis is a negative reappraisal of the
role of the notion of term order which is usually considered a central and
crucial tool in the theory. In fact, as it was already established in the 
computer science context in relation with  termination of algorithms, most of
the properties can be obtained simply consi\-dering
a well-founded order, while  the classical requirement that it be preserved
by multiplication is irrelevant.
The last part of the paper shows how the polynomial basis concepts present in
li\-terature are interpreted in our language and their properties are
consequences of the general results established in the first part of the
paper.\end{abstract}

\maketitle
 
  \section{Introduction}\label{Intro}
  Buchberger reduction was introduced in 1899 by Gordan \cite{Gor1} as a technical tool in his proof of  Hilbert's {\em Basissatz} \cite{Hil} but, at that time, at least the PDE community was aware of the concepts of generic initial ideal introduced in 1896 by Delassus \cite{Del} and of S-polynomials introduced in 1910 by Riquier \cite{Riq2}. This knowledge was summarized by Janet in \cite{Jan1}.
  
  When such theory was independently rediscovered by Buchberger \cite{Bu1,Bu2,Bu3} under the name of Gr\"obner basis, the Pandora box  was opened:   
  Buchberger Theory and Algorithm introduced for polynomial rings over a field \cite{Bu1,Bu2,Bu3}  was extended to polynomial ring over the integers \cite{KrK0}, over Euclidean domains \cite{KrK}, over each ring on which ideal membership is testable and syzygies are computable \cite{Zac}, over domains  \cite{P} and PIRs \cite{M}, to non-commutative
rings which satisfy Poincar\'e-Birkhoff-Witt Theorem \cite{Be}, Lie algebras \cite{AL1,AL2}, solvable polynomial rings \cite{KrW,Kr}, skew polynomial rings \cite{W,Ga1,Ga2,Ga3},  multivariate Ore extensions \cite{P1,P2,BGV,Ore},
other 
algebras
 which satisfy Poincar\'e-Birkhoff-Witt Theorem  \cite{A0,Lev,Lev2},   semigroup rings \cite{Ros,MR1,MR2}, function rings \cite{R2,R3}, non-commutative free algebras \cite{Be}, all effective rings \cite{Bath, Spic},
 reduction rings \cite{St1,St2,St3,St4,Male}, involutive bases \cite{Pom,Z,PomAk,Gerdt, GB1, GB2,BG3,BG4,GB5,S,SB}, marked bases \cite{BCLR,BCR,CR}.

Except \cite{BCLR,BCR,CR} and Gordan\footnote{which actually used the lex ordering induced by $x_1<\ldots<x_n$ without making reference to its semigroup ordering property but only to its well-orderedness.}, all these results make a strong and non-necessary requirement in order to grant termination of   the reduction procedures; in fact they imposed a semigroup ordering on the set of the monomials, i.e. an ordering that
preserves multiplication by variables, while a noetherian well-founded ordering can be  sufficient.

It is true that the results reported by Janet apply the deglex ordering induced by $x_1<\ldots<x_n$ and explicitly assume that the noetherian ordering  preserves multiplication by variables\footnote{Actually Riquier \cite{Riq1,Riq2} applied his theory to a class of orderings which is the classical representation of term ordering introduced by Erd\"os \cite{E} and Robbiano \cite{Robbiano}.}, but their motivations are completely different:
\begin{itemize}
  \item for the researchers developing techniques for solving PDE's multiplication by variable was just an algebraic notation for derivation; and in each calculus course, derivation of a formula by any single variabile is naturally performed by scanning it;
  \item Hilbert's proof of the {\em Nullstellensatz} is done by inductively performing euclide\-an division in the univariate polynomial ring $K[x_1,\ldots,x_{n-1}][x_n]$; while Hil\-bert not even make reference to an ordering of the monomials, it is obvious that a reformulation of Hilbert's reduction in terms of Buchberger's reduction requises the deglex ordering 
  induced by $x_1<\ldots<x_n$ .
\end{itemize}

While the assumption of having a term ordering is obviously justified for historical reasons in the results reported by Janet, we cannot imagine any valid reason for maintaining such an irrelevant assumption in the  research started from the introduction of Groebner bases theory. 

Actually, this assumption hinders the study of Hilbert scheme; it is well-known   \cite{B-M} that deformations of the Groebner basis of an ideal ${\sf I}$ in the polynomial ring ${\mathcal P}$ are a flat family and can thus be applied for studying geometrical deformations of the scheme ${\mathcal X}$ defined by ${\sf I}$. However such families of deformations in general cover only locally closed subschemes of Hilbert scheme and are not sufficient to study neighbourhoods of deformations of ${\mathcal X}$, {\em id est} opens of  Hilbert scheme;
such opens can be obtained instead by considering \cite{BLR} those ideals ${\sf I}'$ of ${\mathcal P}$ which share with 
${\sf I}$ a fixed monomial basis of the quotient ${\mathcal P}/{\sf I}$. In order to determine  the family of all such ideals ${\sf I}'$ of ${\mathcal P}$,
term-ordering free bases of polynomial ideals were introduced, under the label of {\em marked bases} in \cite{BCLR,BCR,CR}.

Following Riquier and Janet,
given a finite set $\mathcal F$ of  polynomials,
they allow to multiply each $f\in \mathcal F$ only by a restricted set of variables ({\em  multiplicative variables} in Janet formulation) or, in general, by an order ideal $\tau_f$ of terms ({\em multiplier set}).
It is then sufficient for them to restrict the requirement of preserving leading terms to such subsets of multipliers for obtaining well-founded orders which are not semigroup ones but however grant a Noetherian 
reduction. Clearly, this does not contradicts  Reeves-Sturmfels Theorem for the elementary reason that the aforementioned theorem requires the application of the whole set of terms as multipliers.

The aim of this paper is to study the main properties  of the consequent Noetherian 
reduction (and its differences with Buchberger reduction); we cover Noetherianity, weak Noetherianity,  confluency, canonical forms; moreover we import in our setting results available within the theories of Greobner bases and of involutiveness as Buchberger's and M\"oller's Criteria and Janet-Schreyer approach for computing resolutions. 

 Fixed the notation (Section \ref{Notat}) and introduced the definition and related notions of {\em reduction structure} (Section \ref{Strutture}), we discuss (Section \ref{MSetsSec}) {\em marked sets} and the associated rewriting rule $\rightarrow$, focusing on its main properties, noetherianity, weak noetherianity, their relation with the orderedness of the related reduction structure  (Sections \ref{Noeth1Sec}, \ref{Noeth2Sec}) and with  Reeves-Sturmfels Theorem (Theorem~\ref{REEVESTURMFELS}), the structure of the related Gr\"obner representation (Proposition~\ref{ScritturaDistinti1} of Section \ref{Noeth2Sec}),  confluency (Section \ref{Conflusect}), criteria for marked bases 
 (Section \ref{sez:disgiunti}) and 
for avoiding useless reductions (Section \ref{Criteriasec});   the functorial description of reduction structures is contained in appendix \ref{AppendixFunct}.  

Next we discuss  stably ordered reduction structures (Section \ref{sez:ordered}). Finally (Sections \ref{casispeciali}, \ref{zerodim}) we cover the most important types of known polynomial bases consistent with a term order reformulating them in our language.

\section{Notations.}\label{Notat}
 
Consider the polynomial ring $$\PP:=A[x_1,...,x_n]=\bigoplus_{d \in \NN}\PP_d$$
in $n$ variables  and coefficients in the base
field $A$. For every set $V\subset \PP$ we denote by$\langle V \rangle$  the $A$-module generated by $V$. When more than one base field $A$ is involved we write $\PP_A$ and  $ {}_A\langle V\rangle$ instead of $\PP$ and  $ \langle V\rangle$.

When an order on the variables comes into play,
we consider $x_1<x_2<...<x_n$.

The \emph{set of terms} in the variables $x_1,...,x_n$ is  
$$\mathcal{T}:=\{x^{\alpha}=x_1^{\alpha_1}\cdots x_n^{\alpha_n},\,(\alpha_1,...,\alpha_n)\in \NN^n \}.$$

For every polynomial $f\in \PP$,
$deg(f)$ is its usual degree and  $deg_i(f)=deg_z(f)$ is
 its degree with respect to the variable $x_i=z$.

Given a term $x^\alpha  \in \mathcal{T}$,  we denote $\vert \alpha\vert:=\deg(x^\alpha)$ and set
$$\max(x^\alpha)=\max\{x_i\ :\  x_i\mid x^\alpha \}   \textit , \, \min(x^\alpha)=\min\{x_i\  :\  x_i \mid x^\alpha\}$$
the maximal and the minimal variable appearing
 in $x^\alpha$ with nonzero exponent.

If $ \{x_{j_1},...,x_{j_r}\}\subset \{x_1,...,x_n\}$, we define 
$$\mathcal{T}[x_{j_1},...,x_{j_r}]:=\{x_{j_1}^{\alpha_{j_1}} \cdots x_{j_r}^{\alpha_{j_r}},\, (\alpha_{j_1},...,\alpha_{j_r})\in \NN^r\}. $$

 For each $p \in \NN$, and for all $V \subseteq \PP$,
$$V_p:=\{f \in V \ :  \ f  \,\textrm{ homogeneous and }  \, \deg(f)=p\};$$
in particular:
$$\mathcal{T}_p:=\{x^\alpha   \in \mathcal{T}\, : \, \deg(x^\alpha)=p\}.$$

We also denote $\mathcal{T}_{\geq p}:=\{x^\alpha   \in \mathcal{T}\, : \, deg(x^\alpha)\geq p\}.$ 

 Once a well-founded order $<$ is fixed in $\TT$ 
then 
each $f\in \PP$  has a unique representation as 
an ordered linear combination of terms $t\in\TT$ with coefficients in $A$:
$$f = \sum_{i=1}^s c(f,t_i) t_i : c(f,t_i) \in A\setminus\{0\}, t_i \in \TT, 
t_1 > \cdots > t_s.$$
\def\lc{\mathop{\rm lc}\nolimits}\
The {\em support} of $f$ is the set $$\supp(f)
:= \{t : c(f,t)\neq 0\}=\{t_1,\ldots ,t_s\};$$ we further denote
${\bf T}(f) := t_1$
the {\em maximal term}\index{maximal!term} of $f$, 
$\lc(f) := c(f,t_1)$ its {\em leading coefficient}\index{leading!cofficient} and 
${\bf M}(f) := c(f,t_1)t_1$ 
its {\em maximal monomial}\index{maximal!monomial}.
 
\def\lc{\mathop{\rm lc}\nolimits}\

For each $f,g\in{\mathcal P}$ such that $\lc(f) = 1 = \lc(g),$
the {\em S-polynomial of $f$ and $g$}  \cite{Bu1,Bu2}\cite[II, Definition~25.1.2.]{SPES} is 
the polynomial
$$S(g, f) := {\frac{\lcm({\bf T}(f),{\bf T}(g))}{{\bf T}(f)}} f -  
{\frac{\lcm({\bf T}(f),{\bf T}(g))}{{\bf T}(g)}} g.$$
 
For an ordered set $F=\{  f_1,...,f_s\}\in \PP $ we denote $Syz(F)$ its \emph{syzygy module }
$$Syz(F)=\{(g_1,...,g_s)\in \PP^s,\, \sum_{i=1}^s  g_i f_i=0 \},$$
all the \emph{syzygies of} $F$ being its elements $(g_1,...,g_s)\in Syz(F)$.

If $I$ is either a monomial ideal
 or a \emph{semigroup ideal}\footnote{
i.e. $\forall x^\eta  \in \mathcal{T},\,  x^\gamma  \in I \Rightarrow x^{\eta+\gamma} \in I $.}
we denote 
by ${\sf N}(I)$ the \emph{order ideal} (or: \emph{normal set})\footnote{Observe that a set of terms ${\sf N} \subset \mathcal{T}$ is an order ideal if
and only if the complementary set
$I:=\mathcal{T}\setminus {\sf N}$ is a semigroup ideal.}
${\sf N}(I):=\mathcal{T}\setminus I$.

\section{Introducing Reduction Structures}\label{Strutture} 
 \begin{definition}\label{RS} 
 A \emph{reduction structure} (RS for short)  $\mathcal{J}$ in $\TT$ is a $3$-tuple  
$$(M, \lambda:=\{\lambda_\alpha\, , \ x^\alpha \in M\}, \tau:=\{\tau_\alpha \, , \ x^\alpha \in M\})$$ 
that satisfies the following conditions
 \begin{itemize}
\item $M$ is a \emph{finite} set of terms; we denote by   $J$ the  semigroup ideal generated by $M$; 
\item  for all $x^\alpha\in M$, $\tau_\alpha \subseteq \mathcal T$ is an order ideal, called  \emph{multiplicative set of} 
 $x^\alpha$, s.t. $\bigcup_{x^\alpha \in M} \cone(x^\alpha)=J $, where $\cone(x^\alpha) :=\{ x^{\alpha+\eta}\ \ \vert\ 	 \  x^\eta \in \tau_\alpha \}$ is the  \emph{cone} of $x^\alpha $; 
\medskip

\item  for all $x^\alpha\in M$,  $\lambda_\alpha$ is a subset of  $ \mathcal T\setminus \cone(x^\alpha)$ that we call  \emph{tail set} of  $x^\alpha$.
 \end{itemize}
 \end{definition}

 \begin{lemma}  \label{primolemma} Let $\mathcal J$ be a RS.
 Then, there is at least a term  $x^\alpha \in M$  s.t.  $\tau_\alpha=\TT$. In particular it holds 
$\mathcal T=\bigcup_{x^\alpha \in M} \tau_\alpha$. 
\end{lemma} 
\begin{proof}
Suppose that the assertion is false and, for each
$x^{\alpha_i}\in M$, choose a term $x^{\eta_i}$ not belonging to   $\tau_{\alpha_i}$.    We denote   $x^\beta $  the product of the terms $x^{\alpha_i+\eta_i}$,  $x^{\alpha_i}\in M$. 
By definition of RS there is a term 
 in $ M$, let it be $x^{\alpha_1}$, whose cone contains $x^{\beta}$, so  $x^{\beta-\alpha_1}$ is a multiple of  $x^{\eta_1}$ and belongs to $\tau_{\alpha_1}$. Since $\tau_{\alpha_1}$
 is an order ideal, it contains also  $x^{\eta_1}$, leading to a contradiction.
\end{proof}

\begin{definition}\label{def:substructure} We will call {\it  substructure} of   $\mathcal J=(M,\lambda , \tau )$ each RS of the form $\mathcal J'=(M,\lambda, \tau')$ s.t. for each  $x^\alpha \in M$ it holds 
$\tau'_\alpha \subseteq \tau_\alpha$.  In this case we will write $\mathcal J' \subseteq \mathcal J$.
\end{definition}

Reduction Structures of the following type will be important in the whole paper
\begin{definition}
A Reduction Structure $\mathcal J$ is:
\begin{itemize}
\item {\it homogeneous} if $\forall x^\alpha\in M$   it holds $\lambda_\alpha \subset \TT_{\vert \alpha\vert}$,
\item with {\it finite tails} if $\forall x^\alpha\in M$ it holds  $\vert  \lambda_\alpha \vert <\infty$,
\item  with {\it reduced tails} if $\forall x^\alpha\in M$ it holds $  \lambda_\alpha\subseteq \cN(J)$, 
\item  {\it coherent with a term order} $\prec$  if $\forall x^\alpha\in M$ and $\forall x^\gamma \in \lambda_\alpha$  it holds $ x^\alpha \succ x^\gamma$, 
\item with {\it maximal cones} if  $\forall x^\alpha\in M$ it holds $  \tau_\alpha=\TT$,
\item  with {\it disjoint cones} if $\forall x^\alpha, x^{\alpha'}\in M$, $x^\alpha\neq  x^{\alpha'}$,  it holds $\cone(x^\alpha)\cap\cone(x^{\alpha'})=\emptyset$,
\item with {\it multiplicative variables} if  for each $x^\alpha\in M$ exists 
$$\mu_\alpha\subseteq \{x_1, \dots, x_n\} \mbox{ s.t. } \tau_\alpha=\TT[\mu_\alpha].$$ 
\end{itemize}
More generally, we will call  $x_i$  { \it multiplicative variable} for   $x^\alpha\in M$ if  $x_i\tau_\alpha \subset \tau_\alpha$.
\end{definition}

As we will see in details in Section \ref{casispeciali}, there are RSs that give the natural framework in which we find Gr\"obner bases and their properties. They are built as follows: $M$ is any finite set of terms;  for each term $x^\alpha \in M$, $\tau_\alpha$ is the whole $\mathcal T$ and $\lambda_\alpha$ is the  sets of   terms lower than $x^\alpha$ w.r.t. a fixed term order.  In the terminology just introduced, these RSs are coherent with a term order,  have multiplicative variables and maximal cones. 

 On the other hand, our definition also includes  strange RSs that cannot be included neither in a standard Gr\"obner framework nor in any other type of  polynomial  bases that (in our knowledge)  are already present  in literature. 

\begin{example} \label{strano}   In $A[x,y]$ let us consider the RS $\mathcal J$ given by
\begin{itemize}
\item $M=\{x^3, xy, y^3\}$; 
\item  $\lambda_{x^3}=\lambda_{y^3}=\{x^2y,xy^2,x^2,xy,y^2,x,y,1\}$, 
$\lambda_{xy}=\{x,y,1\}$
\item       $\tau_{x^3}=\TT[x]$,  $\tau_{xy}=\TT[x,y]$,  $\tau_{y^3}=\TT[y]$.
 \end{itemize}
 This RS is not cosistent with a term order; however
it has
two of the most useful features that we can expect by a polynomial rewriting  rule and that we will discuss in the following sections:  noetherianity and confluence. 
\end{example}
 
\section{Marked sets and rewriting rules}\label{MSetsSec}

\begin{definition}[\cite{RS}]
A \emph{marked polynomial} is a polynomial
 $f\in \PP$ together with a fixed term $Ht(f)$, its \emph{marked term} that appears in $f$ with coefficient $1_A$.
\end{definition} 

We use RSs in order to investigate when and how marked polynomials can be efficiently applied as rewriting rules (for theoretical results on polynomial rewriting rules see \cite{DM,BL,BO}).

\begin{definition}
Given a RS $\mathcal{J}=(M, \lambda, \tau)$,  consider for each $x^\alpha \in M$ a monic marked polynomial  $f_\alpha\in \PP $ s.t. $\Ht(f_\alpha)=x^\alpha$ and $\supp(f_\alpha-x^\alpha)\subset \lambda_\alpha$.  We 
call \emph{marked term}  of $f_\alpha$  such term $x^\alpha$  and {\it tail} of $f_\alpha$ the difference $f_\alpha -x^\alpha$.  

The set  $\mathcal{F}=\{f_\alpha\}_{x^\alpha\in M}$ of polynomials in $\PP$ is called \emph{ marked set } on $\mathcal J$; note that $M$ is indeed the set of the marked terms of $\mathcal{F}$.

We denote by $\tau\mathcal F$ the set $\tau\mathcal F:=\{x^{\eta}f_{\alpha}: x^{\eta} \in \tau_\alpha\},$  by $\langle\tau \mathcal F\rangle$ the $A$-vector space  generated by $\tau\mathcal F$  and by $(\mathcal F)$ the ideal of  $\PP$ generated by $\mathcal F$.
\end{definition}

A key notion in all the theory is the following
\begin{definition}\label{BaseMark} 
We say that  a marked set    $\mathcal F$  over a   RS $\mathcal J$  
is a \emph{ marked basis } on $\mathcal J$ if $\cN(J)$ 
is a free set of generators for   $A[x_1, \dots , x_n]/ ( \mathcal F) $ as $A$-vector space, i.e. if it holds 
$$ ( \mathcal F) \oplus \langle {\cN}(J)\rangle =\PP.$$
\end{definition}

We can associate to a marked set $\mathcal F$ on $\mathcal J$ a reduction procedure  $\rightarrow_{\mathcal F \mathcal J}^+$. 

For $g,h\in \PP$, it holds $g\rightarrow_{\mathcal F \mathcal J}h$ 
iff there are a term $x^\gamma \in \supp(g)$, and an element  $x^\alpha\in M$ s.t.  $x^\gamma=x^{\alpha + \eta} \in  \cone(x^\alpha)$
and $h = g-cx^\eta f_\alpha$, where $c=c(g,x^\gamma)\in A$ is the coefficient of $x^\gamma $ in $g$.

We denote $\rightarrow_{\mathcal F \mathcal J}^+$  its  transitive closure.

We also remark that if $g\in \PP$ and $\supp(g)\in\cN(J)$,  then there is no $h\in\PP,h\neq g$, such that $g\rightarrow_{\mathcal F \mathcal J}^+h$; if this happen, we say that $g$ is {\it reduced}  w.r.t. $J$ or is a $J$-{\it remainder}\footnote{Recall that $J$ denotes the semigroup ideal generated by $M$.}.
 
\begin{definition}  A  {\it  rewriting rule} is a couple $( \mathcal F, \, \rightarrow_{\mathcal F \mathcal J}^+)$, where $\mathcal F$ is a marked set over a RS  $\mathcal J$  and $\rightarrow_{\mathcal F \mathcal J}^+$ is the binary relation defined above.
\end{definition}

\begin{remark} \label{codefinite}     If $\mathcal F=\{ f_\alpha , \ x^\alpha \in M\}$ is a marked set over $\mathcal J=(M,\lambda, \tau)$, then it is also marked over every     RS $\mathcal J'=(M',\lambda', \tau')$ such that  $M'=M$ and $\lambda_\alpha' \supseteq \supp(f_\alpha-x^\alpha)$ for every $x^\alpha \in M$.   From a different point of view, note that   $\mathcal F$ is marked also on every substructure $\mathcal J'$ of $\mathcal J$.

Some notion related to $\mathcal F$  depends on which RS we are considering, while others do not. For instance, it is obvious from the definition that   the notion of marked basis  does not depend on the   RS.  
  On the other hand, a  same set of marked polynomials $\mathcal F$   related to  several  RSs gives rise to essentially different rewriting rules depending  on the set of multiplicative terms. 

If $\mathcal J'=(M,\lambda, \tau')$ is a substructure of $\mathcal J$,  then $\tau'_\alpha \subseteq \tau_\alpha$ for every $x^\alpha \in M$, so that 
$$g\rightarrow_{\mathcal F \mathcal J'}^+ h\quad  \Longrightarrow   \quad  g \rightarrow_{\mathcal F \mathcal J}^+  h  .$$

An interesting example of  RS on which   $\mathcal F$  is marked and also the  rewriting rule is not modified, is 
  $\widetilde{\mathcal J}:=(M,\{\lambda_\alpha \cap\supp (f_\alpha) \}, \tau)$
 The terms of each $\lambda_\alpha$ not appearing in $f_\alpha$ are irrelevant in the reduction steps involving $f_\alpha$. Moreover, they are irrelevant for the steps not involving $f_\alpha$. This yields an advantage: we can work over RSs with finite tails. Anyway, notice that the set of marked sets over $\mathcal J'$ is a proper subset of the analogous over $\mathcal J$.
 \end{remark} 

 While in principle the theory (but not the practice!) of RSs can cover Hironaka Theory and reduction of series in the setting of \cite{7var, Graal} and \cite[II.Hironaka Theorem 24.6.16]{SPES} substituting the Noetherianity assumption with \emph{inflimitedness} 
 
 \cite[II.Definition~24.5.2;IV.Definition~50.3.3]{SPES} we restrict our paper to the polynomial setting ; thus
  {\bf in what follows, each RS will have finite tails}.

\begin{definition} 
If $g \rightarrow_{\mathcal F \mathcal J}^+ l$ and $\supp(l)\subset \cN(J)$,   we write $ g \rightarrow_{\mathcal F \mathcal J}^+ l\downarrow$ and  call  $l$ a \emph{reduced form} (or $J$- {\it remainder}) of  $g$, obtained via $\mathcal F$.   If such a polynomial $l$ exists, we say that $g$ \emph{has a complete reduction w.r.t. }$\mathcal F\mathcal J$. 

Notice that  there could be  several terms in  $\supp(g)\cap J$ and that  each of  them   can belong to several   cones. 
Therefore, the reduction performed on a general polynomial $g$  by a rewriting rule is in general far for being unique nor, in principle, is unique its output, unless \cite{BL,BO} it is both Noetherian and (locally) confluent.
\end{definition}

\section{Noetherianity I: well-founded orders}\label{Noeth1Sec}
In this section we discuss  some relations between the different types of  RSs we have  introduced in relation with the noetherianity of the rewriting rules.  

\begin{definition} \label{defnoetherina}

 We say that 
 $ \rightarrow_{\mathcal F \mathcal J}^+$  is  \emph{noetherian} if 
 there is no inifinite reduction chain
 $$g_1\rightarrow_{\mathcal F \mathcal J}^+g_2\rightarrow_{\mathcal F \mathcal J}^+g_3\rightarrow_{\mathcal F \mathcal J}^+\cdots.$$
 
We call $\mathcal J$  \emph{noetherian}    
 if for each  marked set $\mathcal F$  on $\mathcal J$, the rewriting rule 
 $ \rightarrow_{\mathcal F \mathcal J}^+$ is noetherian. 
The RS $\mathcal J$  is \emph{weakly noetherian}  
 if it has a noetherian substructure.
 \end{definition}

If $\mathcal J$  is weakly noetherian, each polynomial $g $ has a complete reduction
$ g \rightarrow_{\mathcal F \mathcal J}^+ l\downarrow$,  though there could be also  infinite sequences of base steps of reduction starting on $g$.

Here an example of a RS that is not  weakly noetherian (nor a fortiori noetherian).

\begin{example}\label{EXORSaltro}

Let us consider  the RS $\mathcal J$ given by 

\begin{itemize}

\item $M=\{xy,x^3, y^3\}$;

\item $\tau_{xy}=\TT[x,y]$, \ $\tau_{x^3}=\TT[x]$,\ $\tau_{y^3}=\TT[y]$;

\item $\lambda_{xy}=\{x^{2},y^2\}$,\ $\lambda_{x^3}=\lambda_{y^3}=\emptyset$.

\end{itemize}

As the cones are disjoint, $\mathcal J$ has no proper substructure, hence it is sufficient to show that $\mathcal J$ itself  is not noetherian.

Let us consider  the marked set $\mathcal F=\{f_{xy}=xy-x^2-y^2,f_{x^3}=x^3,f_{y^3}=y^3\}$ over $\mathcal J$.  

We obtain an infinite reduction chain as follows:

\begin{align*}
 & x^2y\longrightarrow_{\mathcal{FJ}} x^2y-xf_{xy}= 
 x^3+xy^2 \longrightarrow_{\mathcal{FJ}} x^3+xy^2-f_{x^3}= 
 xy^2 \longrightarrow_{\mathcal{FJ}}  xy^2-yf_{xy}= \\
& x^2y+y^3  \longrightarrow_{\mathcal{FJ}} x^2y+y^3-f_{y^3}=x^2y \longrightarrow_{\mathcal{FJ}} x^2y-xf_{xy}\dots .
\end{align*}

Note that  at each  step of every possible sequence of reductions of $x^2y$ we  find  a  polynomial of the type  $x^2y+nx^3+my^3$ or $xy^2+nx^3+my^3$ with $n,m\in \ZZ$, and none of them is a $J$-remainder.

\end{example}

Here an example of a  weakly noetherian  RS that is  not noetherian.

\begin{example}\label{EXORS2} 

Let us consider  the RS $\mathcal J$ given by 

\begin{itemize}

\item $M=\{xy,y^2,x^3, y^3\}$;

\item $\tau_{xy}=\tau_{y^2}=\tau_{x^3}=\tau_{y^3}=\TT[x,y]$;

\item$\lambda_{xy}=\{x^{2},y^2\}$,\  $\lambda_{y^2}=\lambda_{x^3}=\lambda_{y^3}=\emptyset$.

\end{itemize}

Every marked set over $\mathcal J$ has the shape

$$\mathcal F_{a,b}= \{f_{xy}=xy-ax^2-by^2, f_{y^2}=y^2,f_{x^3}=x^3,f_{y^3}=y^3\},  \   a,b\in A.$$ 

The RS $\mathcal J$  is not noetherian since reducing  $x^2y$ with respect to the marked  set $\mathcal F_{0,0}$ we may obtain the same infinite sequence of steps  described in Example \ref{EXORSaltro}.

However, in this case for  every polynomial there are also reductions   leading  to a $J$-remainder, since $\mathcal J$  has  for instance the  noetherian substructure $\mathcal J'$ given by   $\tau_{xy}'=\{1, x\}$  and  $\tau_{y^2}'=\tau_{x^3}'=\tau_{y^3}'=\TT[x,y]$. 

In fact, for every $\mathcal F_{a,b}$ the reduction procedure  $\rightarrow_{\mathcal F_{a,b} \mathcal J'}^+ $ returns after the only possible first  step of reduction the $J$-remainder of $xy$ (it is $ax^2+by^2$) and of every monomial $v$   that is multiple of either $y^2$ or $x^3$ (it is $0$).   Moreover, the only possible sequences of reduction of $x^2y$ are 

$$x^2y \rightarrow_{\mathcal F_{a,b} \mathcal J'}^+ ax^3+bxy^2 \rightarrow_{\mathcal F_{a,b} \mathcal J'}^+ ax^3 \rightarrow_{\mathcal F_{a,b} \mathcal J'}^+0\downarrow$$ $$  x^2y \rightarrow_{\mathcal F_{a,b} \mathcal J'}^+ ax^3+bxy^2 \rightarrow_{\mathcal F_{a,b} \mathcal J'}^+  b xy^2 \rightarrow_{\mathcal F_{a,b} \mathcal J'}^+  0\downarrow. $$

\end{example}

\begin{lemma}\label{implicazionibanali}  Let $\mathcal J$ be a RS. Then
\begin{enumerate}[(i)]
\item \label{implicazionibanali1} if $\mathcal J$  is noetherian, then it is also weakly noetherian;
\item   \label{implicazionibanali2} if $\mathcal J$ has disjoint cones, then also the converse  of (i) holds true;
\item \label{implicazionibanali3} if $\mathcal J'\subseteq \mathcal J$ and $\mathcal J$ is noetherian, then also $\mathcal J'$ is noetherian;
\item \label{implicazionibanali4}   if $\mathcal J'\subseteq \mathcal J$ and $\mathcal J'$ is weakly  noetherian, then $\mathcal J$ is weakly  noetherian.
\end{enumerate}
\end{lemma}
\begin{proof} 
All these properties are  trivial consequences of the definitions. We only observe for \eqref{implicazionibanali2} that a RS with disjoint cones has no proper substructures.
\end{proof}

In order to find some effective way  to check the noetherianity of a RS, we now exploit arguments and results concerning the termination of algorithms based on rewriting rules,  that have been developed mainly in the  computer science context. They state  a closed relation between the noetherianity and the presence of a suitable  well-founded order. 

We recall  that an order  $\Prec$ on a set $W$ is called  \emph{well-founded} if each nonempty subset of $W$ contains minimal elements.

\begin{definition}\label{ORS}
We say that a RS  $\mathcal J$ is \emph{ordered}
 if there is a well-founded order $\Succ $ on   $  \mathcal T $ s.t.
 $$ \forall \   x^\alpha \in M,  \ x^\gamma \in \lambda_\alpha,  \ x^\eta \in \tau_\alpha  \hbox{  it holds   } x^{\alpha+\eta} \Succ x^{\gamma+\eta}.$$
\end{definition}

All the RSs coherent with a term order $\prec$ are obviously ordered. However there are ordered RSs that are not coherent with a term order; an easy example is the following.

\begin{example}\label{EXORS}
Let us consider  the RS $\mathcal J$ given by 
\begin{itemize}
\item $M:=\{x^3, xy,x y^2,y^3\}$;
\item $\tau_{x^3}=\tau_{xy}=\tau_{xy^2}:=\TT[x]$,
$\tau_{y^3}:=\TT[x,y]$;
\item $\lambda_{x^3}:=\lambda_{xy^2}:=\lambda_{y^3}:=\emptyset$,
$\lambda_{xy}:=\{x^{2},y^2\}$.
\end{itemize}
We prove  $\mathcal J$ to be ordered, by considering the well-founded order $<$ defined by  
$m_1\Succ  m_2$ if and only if $m_1=x^ay$ and either $m_2=x^{a+1}$ or $m_2=x^{a-1}y^2$ for some positive integer $a$. Of course there is no term ordering $\succ$ such that both $xy\succ x^2$ and $xy\succ y^2$.
\end{example}

\begin{example} \label{strano2}  The RS of Example \ref{strano} is ordered by the  well founded order $\Prec$ that we obtain refining the one given by the degree in the following way 
$$\forall \mathbf u, \mathbf v \in \mathcal T[x,y]_d: \quad  \mathbf u \Prec \mathbf v \Longleftrightarrow   \mathbf u \notin \{ x^d,y^d\}, \ \mathbf v \in \{ x^d,y^d\}.  $$
 
\end{example}
\medskip

We would like  to connect this definition of ordered RS to the rewriting rules on it. To this aim,  we  adapt to our situation a more general construction presented by Dershowitz and  Manna in \cite{DM} and extend any order $\Prec $ on  the set of monomials $\mathcal T$ to an order $\PPrec$  on the set of   polynomials $\mathcal P$,  by setting for every pair  $f,g \in \mathcal  P$, $f \SSucc  g$ if and only if
  
$$\supp(f)\neq \supp(g) \hbox{ and } \forall   m \in \supp(g)\setminus \supp(f)  \,   \exists m' \in \supp(f)\hbox{ s.t. } m'\Succ m .$$

\begin{theorem}\label{wellfounded} \cite{DM}  $(\mathcal T, \Prec)$ is well-founded if and only if $(\mathcal P , \PPrec)$ is.
\end{theorem}
It is quite obvious that for every marked set $\mathcal F$ on a RS $\mathcal J$ ordered by $\Prec $,  
$f\rightarrow_{\mathcal F \mathcal J}g$
implies $\supp(f) \SSucc \supp(g)$.  We can then reformulate in our framework a well know results by Z. Manna and S. Ness concerning the termination of programs (\cite{MN},\cite{DM}).

\begin{theorem}\label{implicazioni} Let $\mathcal J$ be a RS. Then
  $$\mathcal J \hbox{ is ordered  } \ \Longleftrightarrow  \mathcal J  \hbox{ is   noetherian}.$$
\end{theorem}

 Due to the above result, {\bf in the following we consider noetherian and ordered as synonyms  for what concerns the RSs.} Therefore, to every noetherian RS 
we associate a well founded ordering $\Prec$ on  $\TT$   and its extension $\PPrec$ on $\PP$.
\medskip

We conclude this section with a reformulation of a well known result by Reeves and Sturmfels.
They assume to have a finite set $\mathcal F$ of marked polynomials and they say that $\mathcal F$ is marked {\em coherently} if there is a term ordering $\prec$ such that for each $f=Ht(f)-\sum_{j=1}^s c_j t_j\in \mathcal F, c_j\neq 0$ it holds  $t_j\prec Ht(f)$ for each $j$
and they prove that

\begin{theorem}[Reeves-Sturmfels, \cite{RS}[Theorem 1.1]]\label{REEVESTURMFELS}
A set $\mathcal F\subset{\mathcal P}$ of marked polynomials is  marked coherently if and only if the reduction relatation modulo $\mathcal F$ is Noetherian, i.e. every sequence of reductions modulo $\mathcal F$ terminates.
\end{theorem}

The core of their argument is the following lemma
\begin{lemma}[Reeves-Sturmfels, \cite{RS}[Lemma 2.1]]\label{REEVESTURMFELSLEMMA}
 Let $\mathcal F=\{f_1,\ldots,f_k\}\subset{\mathcal P}$ be marked incoherently. Then there exists a reduction sequence  modulo $\mathcal F$ which does not terminate.
\end{lemma}

Since Buchberger reduction terminates under any term ordering, the proof of the lemma is all one needs for proving the theorem for finite sets. The extension to the case of infinite sets requires 
either comminatorial tools (Helly's Theorem) or the Compacteness Theorem of First-Order Logic\footnote{{\em One writes down an infinite set $S$ of first-order sentences that asserts that ``$\prec$'' 
is an admissible term order extending the order relations between terms specified by $\mathcal F$. 
By the hypothesis of Lemma~\ref{REEVESTURMFELSLEMMA}, every finite subset of $S$ is coherent, 
hence $S$ is coherent \cite{RS}[ppg.276-7].} }.

In their proof of lemma~\ref{REEVESTURMFELSLEMMA}, they represent each marked polynomial  
$f_i := Ht(f_i)-\sum_{j=1}^{s_i} c_j t_j$ as
$$f_i := x^{\alpha_i}-\sum_{j=1}^{s_i} c_j x^{\alpha_i+\gamma_{ij}}$$
via suitable distinct non-zero vectors 
$$\gamma_{ij}=\left(\gamma_{ij1},\ldots, \gamma_{ijn}\right)\in{\mathbb{Z}}$$ and they suppose that the marking $Ht(f_i) = x^{\alpha_i}$ is incoherent,i.e. there is no admissible term order $\prec$ such that $Ht(f_i) = x^{\alpha_i} ={\bf T}_\prec(f_i)$ for all $i, 1\leq i\leq k$.

This implies, as a direct consequence of Linear Programming Duality, the existence of a non-zero, non-negative 
integer vector
$$y=\left(\bar y_1,\ldots, \bar y_n,y_{11},\ldots,y_{1{s_1}},\ldots,y_{ij},\ldots,y_{k{s_k}}\right)$$
which satisfies
$$\left(\bar y_1,\ldots, \bar y_n\right)^T=\sum_{i=1}^k\sum_{j=1}^{s_i} y_{ij}\gamma_{ij}^T.$$

If, there is a $y$ with $\left(\bar y_1,\ldots, \bar y_n\right)^T=0$ they choose any such solution which further minimalizes
$$N:=\sum_{i=1}^k\sum_{j=1}^{s_i} y_{ij};$$ if all solutions are such that
$\left(\bar y_1,\ldots, \bar y_n\right)^T=\sum_{i=1}^k\sum_{j=1}^{s_i} y_{ij}\gamma_{ij}^T>0$ they choose, among all solutions, one which minimalizes
$N:=\sum_{i=1}^k\sum_{j=1}^{s_i} y_{ij}$.

Then they consider the monomial
$$x^{\beta} =x^{\beta^{(0)}}= \prod_{i=1}^k Ht(f_i)^{\sum_{j=1}^{s_i} y_{ij}}$$
and construct proper non negative integer vectors $\beta^{(r)}, 0\leq r\leq N$ such that $x^\beta\mid x^{\beta^{(r)}}$ for each $r$.

Combinatorial arguments allow them to prove that for each $r,1\leq r\leq N$ there is an index $i_r\in\{1,\ldots,h\}$ and a term $t_r\in\TT$ such that
$x^{\beta^{(r-1)}}=t_r Ht(f_{i_r})$ so that denoting
$g_0:=x^{\beta}$ and, inductively
$g_i:=g_{i-1}-c(g_{i-1},\beta^{(r-1)})t_r f_{i_r}$

A further combinatorial argument grants them that for each index $c(g_{i-1},\beta^{(r-1)})\neq 0$ so that 
$g_N=cx^{\beta^{(N)}}+g$ with $c\neq 0$.

Then they need to consider the two different cases.
If $\beta^{(N)}=\beta$ then we have the infinite reduction
$$x^{\beta}\rightarrow cx^{\beta}+g\rightarrow c^2x^{\beta}+(1+c)g\rightarrow\cdots\rightarrow c^nx^{\beta}+\left(\sum_{i=1}^n c^{i-1}\right)g\ldots$$
If the non-negative vector $\bar\beta:=\beta^{(N)}-\beta$ is non zero,
then we have the infinite reduction
$$x^{\beta}\rightarrow cx^{\beta^{(N)}}+g\rightarrow c^2x^{2\beta^{(N)}-\beta}+(1+cx^{\bar\beta})g\rightarrow\cdots\rightarrow 
c^{n+1}x^{\beta^{(N)}+n\bar\beta}+\left(\sum_{i=0}^{n}c^{i}x^{i\bar\beta}\right)g\ldots$$

In our notation this non Noetherian RS can be described as $\rightarrow_{\mathcal F \mathcal J}^+$
with $\mathcal{J}=(M, \lambda, \tau)$, 
$$M=\{x^{\alpha_i}, 1\leq i\leq k\}, \lambda_{\alpha_i}:=\{x^{\alpha_i+\gamma_{ij}}, 1\leq j\leq s_i\}, \tau_{\alpha_i}=\TT.$$

  In the following example we  present a RS
which  is {\em not} coherent with a term order while noetherian.   

By any similar RS, we can obtain examples of  weakly noetherian RSs with maximal cones, though non-coherent with a term order. Indeed, if   $\mathcal J'=(M',\lambda' , \tau' )$   is a noetherian, then  $\mathcal J =(M=M',  \lambda =\lambda' , \{\tau_\alpha=\TT\})$, of which  $ \mathcal J' $ is a substructure, is  weakly noetherian and has maximal cones.

\begin{example}\label{perReevesSturm}
In $A[x,y]$ we consider
\begin{itemize} 
\item $M=\{xy,x^3,y^3, xy^2,x^2y^2\}$; 
\item  $\tau_{xy}=\tau_{x^3}=\TT[x]$, \  \ $\tau_{y^3}=\tau_{xy^2}=\TT[y]$,  \  \ $\tau_{x^2y^2}=\TT[x,y]$;
\item $\lambda_{xy}=\{x^2,y^2\},$  \  \ $\lambda_{x^3}=\lambda_{y^3}=\lambda_{xy^2}=\lambda_{x^2y^2}=\emptyset.$
\end{itemize}

Let us consider the marked set $\mathcal{F}=\{f_{xy},f_{x^3},f_{y^3},f_{xy^2}\}$; while the marked polynomials $f_{x^3},f_{y^3},f_{xy^2}$  are necessarily  monomials, for $xy$ we have to fix a polynomial with the shape $xy-ax^2-by^2, a,b\in K$; the reduction we are discussing assume $a\neq0\neq b$ but is a trivial task to check that our  claim  apply also when either $a=0$ and/or $b=0$.
The RS $\mathcal{J}=(M, \lambda, \tau)$ is trivially {\em non}-coherent with a term order but is noetherian.

In fact: 
\begin{itemize}
\item if $\upsilon\in\cone(x^3)\cup\cone(y^3)\cup\cone(xy^2)\cup\cone(x^2y^2)$, trivially $\upsilon\rightarrow 0\downarrow$;
\item $xy\rightarrow_
{\mathcal F \mathcal J} ax^2+by^2\downarrow \in\langle \cN(J)\rangle$
\item $x^2y=x(xy)\rightarrow_{\mathcal F \mathcal J} x(ax^2+by^2)=ax^3+bxy^2\rightarrow_{\mathcal F \mathcal J}^+ 0\downarrow$
\item $x^{i+3}y=x^{i+2}(xy)\rightarrow_{\mathcal F \mathcal J} x^{i+2}(ax^2+by^2)=ax^{i+1}\cdot x^{3}+by\cdot x^{i}\cdot x^2y^2\rightarrow_{\mathcal F \mathcal J}^+ 0\downarrow, i\geq 0$.
\end{itemize}

Note that the example does not contradict Reeves-Sturmfels Theorem for the simple reason that the cones are {\em not} maximal.
\end{example}

\section{Noetherianity II: lower representations of polynomials}\label{Noeth2Sec}

In this section we relate the reduction of  a polynomial $g$ by a given marked set $\mathcal F$ and its linear/polynomial representation in terms of $\tau \mathcal F$.  We recall that for a given marked set $\mathcal F$ over a RS  $\mathcal J=(M, \lambda, \tau)$, we denote by $\tau \mathcal F$ the set of polynomials $x^\gamma f_\alpha$ with $f_\alpha \in \mathcal F$ and $x^\gamma \in \tau_\alpha$,  and  by $\langle \tau \mathcal F\rangle$ the $A$-module  generated by $\tau \mathcal F$.  Moreover,   $J$ denotes the semigroup ideal generated by $M$ and  $\cN(J)$ the order ideal $\TT \setminus J$.

\begin{definition} \label{representation} Let $\mathcal F$ be a marked set over a RS $\mathcal J$ and let $g$ be any polynomial in $\langle \tau \mathcal F\rangle$. 

If $g=\sum_{i=1}^r c_i x^{\eta_i}f_{\alpha_i}$ with $c_i \in A$ and $x^{\eta_i}f_{\alpha_i}$ distinct elements of $ \tau \mathcal F$, we say that the writing $\sum_{i=1}^r c_i x^{\eta_i}f_{\alpha_i}$ is a \emph{representation of $g$   by} $  \tau \mathcal F $.

If,  moreover, $\mathcal J$ is noetherian with well founded ordering  $\Prec$ and  $x^\delta$ is any term, we say that  a representation   $g=\sum_{i=1}^r c_i x^{\eta_i}f_{\alpha_i}$   by $\tau \mathcal F$ is  a $x^\delta$-{\it{lower representation}}  ($x^\delta-\LR$  for short) and,  respectively, a  $x^\delta$-{\it{strictly lower representation}}   ($x^\delta-\SLR$  for short)  if, for every $i=1, \dots, r$,   it holds $x^{\eta_i+\alpha_i}\Preceq x^\delta$  and respectively $x^{\eta_i+\alpha_i}\Prec x^\delta$.
\end{definition}

We observe that, as an obvious consequence of the definition of reduction procedure, if $g \rightarrow_{\mathcal F \mathcal J}^+ h$, then $g-h$ has a representation by $\tau \mathcal F$ given by the steps of the reduction (summing up the coefficients of each  element of $\tau \mathcal F$ used more than once during the reduction).

\begin{proposition}\label{ScritturaDistinti1} 
Let $\mathcal F$ be a marked set over a weakly noetherian RS   $\mathcal J$ and let  $g \in \PP$. 
\begin{enumerate}[i)]
\item  \label{ScritturaDistinti1_i}  There exists a reduced  form $l$ of $g$ obtained by $\mathcal F$ and  $g-l$ has a representation by $\tau \mathcal F$.
\item \label{ScritturaDistinti1_ii}  If $\mathcal J$ has disjoint cones,   then   there is  only   one polynomial  $l $  (the {\em canonical form} of $g$)  with    $\supp ( l) \subset \cN(J)$ and $g-l \in \langle \tau \mathcal F\rangle $; moreover,  there is a unique  representation  of   $g-l$ by $\tau \mathcal F$. 

\item  \label{ScritturaDistinti1_iii} If   $\mathcal J$ is   noetherian (with well-founded order  $\Prec$) and,  for a reduced polynomial $l$,   $g-l$ has a representation $\sum_{i=1}^r c_i x^{\gamma_i}f_{\alpha_i}$ by $\tau \mathcal F$ with 
  all distinct heads $x^{\gamma_i+\alpha_i}$, then  $g \rightarrow_{\mathcal F \mathcal J}^+ l\downarrow$  and, for each $i$,
$x^{\gamma_i+\alpha_i}  \Preceq x^\delta$  for some  $x^\delta \in \supp (g)$.

Vice versa,  from
 $g \rightarrow_{\mathcal F \mathcal J}^+ l\downarrow$  one deduces that    $g-l$ has a representation $\sum_{i=1}^r c_i x^{\gamma_i}f_{\alpha_i}$  by $\tau \mathcal F$   with 
  all distinct heads  s.t. for each 
 $i$ it holds  $x^{\gamma_i+\alpha_i}  \Preceq x^\delta $ for some  $x^\delta \in \supp (g)$.
\item  \label{ScritturaDistinti1_iv} In the same hypotheses and setting of \ref{ScritturaDistinti1_iii}), if  $g$ is a term $x^\delta$, then $x^\delta-l$ has a $x^\delta-\LR$ by $\tau \mathcal F$. 
 
\end{enumerate}
\end{proposition}

\begin{proof}
i)  follows from the definition of $  \rightarrow_{\mathcal F \mathcal J}^+ $ and the weak noetherianity of $\mathcal J$.

In order to prove ii) we  observe that $\mathcal J$ is in fact noetherian, since a RS with disjoint cones has no proper substructures. 

 Consider two reduced polynomials $l,l'$  such that $g-l,g-l' \in \langle \tau \mathcal F \rangle$ and take  some representations $g-l=\sum_{i=1}^r c_i x^{\gamma_i}f_{\alpha_i}$ and     $g-l'=\sum_{i=1}^r d_i x^{\gamma_i}f_{\alpha_i}$ in $\tau \mathcal F$;
we may suppose that the indices of the two summations are the same, possibly adding some zeroes.

We have then $ l-l'=\sum_{i=1}^r (d_i-c_i) x^{\gamma_i}f_{\alpha_i}$ and we deduce that    $c_i=d_i$ for $i=1, \dots, r$.  If, in fact, this were not true, we could  choose a maximal element in the set $\{x^{\gamma_i+\alpha_i}, i=1,\dots, r, c_i-d_i\neq 0\}$: suppose it  is $x^{\gamma_1+\alpha_1}$. 
Then   $x^{\gamma_1+\alpha_1} $  appears in the support of  $\sum_{i=1}^r (d_i-c_i) x^{\gamma_i}f_{\alpha_i}$: indeed this term is different from $x^{\gamma_i+\alpha_i} $  for  $i=2, \dots, r$, since by hypothesis $\mathcal J$ has  disjoint cones, and it does not appear in the support of   $x^{\gamma_i}f_{\alpha_i} -x^{\gamma_i+\alpha_i} $
 for some $i=1, \dots r$, by maximality. We get then a contradiction since the support of $l-l'$ is contained in $\cN(J)$.  Then $c_i=d_i$ and $l=l'$.

In order to prove iii), we proceed by induction on the number  $r$ of the summands. \\
If $r=1$ then $g=l+c_1x^{\gamma_1}f_{\alpha_1}$, and $x^{\gamma_1+\alpha_1}$ necessarily appears in  $\supp(g)$, 
since it cannot coincide neither with a term in the support of $l$ nor with a term of $x^{\gamma_1}f_{\alpha_1}-x^{\gamma_1+\alpha_1}$.
We can get $l$ from $g$ via a base reduction step on the term   $x^{\gamma_1+\alpha_1}$ using $f_{\alpha_1}$. \\
Setting $x^\delta:= x^{\gamma_1+\alpha_1}$, we trivially have 
$x^{\gamma_1+\alpha_1} \Preceq   x^\delta\in\supp (g).$
\\
Moreover, each term  $x^\beta$ in the support of
$c_1x^{\gamma_1}f_{\alpha_1}$ satisfies 
$x^\beta\Preceq  x^\delta$ since each term
$x^{\gamma}\in\supp(f_{\alpha_1}\setminus\{x^{\alpha_1}\})$ satisfies
$ x^{\gamma_1+\gamma}\Prec  x^{\gamma_1+\alpha_1}\Preceq   x^\delta$.
\\
Suppose by inductive hypothesis that the assertion is true in the case in which we have   $r-1$ summands.

We can suppose that  $x^{\gamma_r+\alpha_r}$ is maximal in the set   $\{x^{\gamma_i+\alpha_i},\, i=1,...,r\}$ 
and so it is also maximal in   $\{x^\epsilon \ \vert \   x^\epsilon \in  \supp (x^{\gamma_i}  f_{\alpha_i}), i=1, \dots, r \}$.  

Then $x^{\gamma_r+\alpha_r}$ appears in the support of $ \sum_{i=1}^r c_i x^{\gamma_i}f_{\alpha_i}$
and so also in the support of  $g$ (remember that $\supp(l)\subset \cN(J)$). 
\\
We execute the first reduction step on   $g$ choosing exactly that term and setting   $g'=g-c_rx^{\gamma_r}f_{\alpha_r}$. 
\\
Setting $x^\delta:=x^ {\gamma_r+\alpha_r}$, we trivially have, for each  $i$,
$ x^{\gamma_i+\alpha_i}\Preceq 
 x^{\gamma_r+\alpha_r} =
 x^\delta\in\supp (g).$
\\
Thus we obtain   $g'-l=\sum_{i=1}^{r-1} c_i x^{\gamma_i}f_{\alpha_i}$ and we conclude by inductive hypothesis.
\\
The converse statement immediately follows from the fact that   $\mathcal J$ is ordered and also from the hypothesis.

iv) is  a consequence of iii) and of  Definition \ref{representation}.
\end{proof}

\begin{corollary} \label{sommasemplice}  Let  $\mathcal F$ be a marked set over a weakly noetherian RS  $\mathcal J$.  Then $$\langle\tau \mathcal F\rangle +\langle \cN(J)\rangle=\PP.$$ 

If, moreover,  $\mathcal J$ has disjoint cones,  then $$\langle\tau \mathcal F\rangle \oplus\langle \cN(J)\rangle=\PP.$$

 In particular,  take  $x^\eta\in \TT$, $g, l\in \PP$   s.t.   $\supp(l)\subseteq \cN(J)$ and $ x^\gamma\Preceq  x^\eta$, for every $x^\gamma \in \supp(g)$.  Then 
$$ g-l \in \langle \tau\mathcal F \rangle \iff g \rightarrow_{\mathcal F \mathcal J}^+ l\downarrow \iff  g-l \hbox{ has a }  x^\eta-\LR \hbox{ by }  \tau\mathcal F . $$
\end{corollary} 
\begin{proof}
The first assertion comes from Proposition  \ref{ScritturaDistinti1}. Indeed,   $\forall g \in \PP$,   from $g=\sum_{i=1}^r c_i x^{\gamma_i}f_{\alpha_i}+ l$ with  $x^{\gamma_i}f_{\alpha_i}\in \tau\mathcal F$  and  $\supp ( l) \subset \cN(J)$ we deduce $g\in  \langle \tau \mathcal F \rangle +\langle \cN(J)\rangle$.  So $\langle\tau \mathcal F\rangle +\langle \cN(J)\rangle\supseteq \PP$. The other implication is obvious.

For the second assertion it is then sufficient to prove that $\langle\tau \mathcal F\rangle \cap  \langle \cN(J)\rangle=0$ and this comes from Proposition  \ref{ScritturaDistinti1}  {\it  \ref{ScritturaDistinti1_ii})}.

Now we prove the last assertion.  If $g-l \in \langle \tau\mathcal F \rangle $, by  \ref{ScritturaDistinti1}  {\it  \ref{ScritturaDistinti1_ii})}, then $g-l$ has a unique representation  $\sum_{i=1}^r c_i x^{\gamma_i}f_{\alpha_i}$  by $\tau \mathcal F$; as $\mathcal J$ has disjoint cones, the heads $x^{\gamma_i+\alpha_i}$ are distinct. By \ref{ScritturaDistinti1}  {\it  \ref{ScritturaDistinti1_iii})} we obtain $g \rightarrow_{\mathcal F \mathcal J}^+ l\downarrow  $ and    $x^{\gamma_i+\alpha_i} \Preceq x^\delta$ for some $x^\delta \in \supp (g)$; then for every $x^\epsilon $ in the support of $x^{\gamma_i}f_{\alpha_i}$ it holds $x^\epsilon   \Preceq x^\eta$, namely  $\sum_{i=1}^r c_i x^{\gamma_i}f_{\alpha_i}$ is a $x^\eta-\LR$.

The other implications are obvious.
\end{proof}

\begin{corollary}\label{reducedtails}
 Let  $\mathcal J=(M, \lambda, \tau)$ be a  noetherian RS  $\mathcal J$.  

Then there is a noetherian RS   $\mathcal J^{Red}=(M, \lambda^{Red}, \tau)$  with reduced tails,   such that for every $\mathcal J$-marked set $\mathcal F$ there is a ${\mathcal J^{Red}}$-marked set  $\mathcal F^{Red}$ that satisfies    $\langle \tau \mathcal F\rangle=\langle \tau\mathcal F^{Red}\rangle$. Moreover,   $\mathcal F$ is a marked basis iff $\mathcal F^{Red}$ is.

If  $\mathcal J$ is also confluent, then    $\mathcal F^{Red}$  is unique and 
 $\forall g\in \mathcal P$
$$g \rightarrow_{\mathcal F \mathcal J}^+ l\downarrow \iff  
g \rightarrow_{\mathcal F^{Red}\mathcal J^{Red}}^+ l\downarrow .$$
\end{corollary}

\begin{proof}
Assume that  $\mathcal J$ is ordered  by $\Prec$  (Theorem \ref{implicazioni}).
For every $x^\alpha \in M$ we choose as $\lambda^{Red}_\alpha$ the support of any  polynomial $\ell_\alpha$ such that $x^\alpha-f_\alpha \rightarrow_{\mathcal F \mathcal J'}^+  \ell_\alpha$ and set $\mathcal F^{Red}=\{f^{Red}_\alpha:=x^\alpha-\ell_\alpha \ |  \ x^\alpha \in M\}$. 
 
 Let us assume that  $\langle \tau \mathcal F\rangle\neq \langle \tau\mathcal F^{Red}\rangle$ and consider a minimal element $x^{\eta+\alpha} \in J$ such that $x^\eta \in \tau_\alpha$ and either $x^{\eta}f_\alpha \notin\langle \tau\mathcal F^{Red}\rangle$ or    $x^{\eta}f^{Red}_\alpha \notin\langle \tau\mathcal F\rangle$. Therefore,   if $x^\delta \in \tau_\beta$ and $x^{\delta+\beta} \Prec x^{\eta+\alpha} $, then both  $x^\delta f_{\beta}\in \langle \tau\mathcal F^{Red}\rangle$ and $x^\delta f^{Red}_{\beta}\in \langle \tau\mathcal F \rangle$. 
 
 By  Proposition \ref{ScritturaDistinti1} \eqref{ScritturaDistinti1_ii},   the difference   $x^\eta f^{Red}_\alpha-x^\eta f_\alpha$ has a $x^{\eta+\alpha}$-SLR $ \sum c_i x^{\eta_i} f_{\alpha_i}$  in $ \tau \mathcal F$;  by the assumption every $x^{\eta_i}f_{\alpha_i}$ also belongs to $ \langle \tau\mathcal F^{Red}\rangle$. 
Then we get a contradiction, since   $x^\eta f^{Red}_\alpha=x^\eta f_\alpha +\sum c_i x^{\eta_i} f_{\alpha_i} \in  \langle \tau\mathcal F\rangle$ and  $x^\eta f_\alpha=x^\eta f_\alpha^{Red} -\sum c_i x^{\eta_i} f_{\alpha_i} \in  \langle \tau\mathcal F^{Red}\rangle$.

As a consequence      $( \mathcal F)=(\mathcal F^{Red})$, so that  $\mathcal F$ is a   basis iff $\mathcal F^{Red}$ is by Corollary \ref{sommasemplice}.

The other assertions are direct consequence of the above construction.
\end{proof}
 
In what follows, we will use the second assertion of Proposition~\ref{ScritturaDistinti1} {\it (\ref{ScritturaDistinti1_iii} )}. Indeed, if one wants to use induction in proofs, it will be useful   to   consider the fact that not only a certain polynomial  $g$ is in $\langle \tau \mathcal F \rangle$, but also that $g$ can be written as a linear combination of elements in $\tau \mathcal F$ whose heads satisfy the property  underlined in  {\it (\ref{ScritturaDistinti1_iii} )}.\\ 
 
The following two examples show that the hypotheses of the various points of Pro\-position~\ref{ScritturaDistinti1} are necessary. Point ii) does not necessarily hold if $\mathcal J$ has non-disjoint cones.  Moreover, the conditions   $g-l\in \langle \tau \mathcal F\rangle $ and  $\supp(l)\subset \cN(J)$ do not necessarily imply that $g \rightarrow_{\mathcal F \mathcal J}^+ l$. 

\begin{example}
In $A[x,y]$, let   $\mathcal J =(M=\{x^2,xy\}, \{\lambda_{x^2}=\lambda_{xy}=\{x\}\} ,  \{\tau_{x^2}=\TT,  \tau_{xy}=\{y^k,xy^k, \ k\in \NN\} \})$; notice that $\mathcal J$ is noetherian, since it is coherent with any  degree compatible term order,  and it has not   disjoint cones  ($x^2y=x^2\cdot y =xy\cdot x$). Let moreover $\mathcal{F}=\{f_{x^2}=x^2-x,f_{xy}=xy\}$ and $g=x^2y-xy$. For each reduction  $g \rightarrow_{\mathcal F \mathcal J}^+ l\downarrow$ we have $l=0$, but $g$  has two representations of the form of Proposition \ref{ScritturaDistinti1}    {\it  \ref{ScritturaDistinti1_ii})}: $g=yf_{x^2}=xf_{xy}-f_{xy}$.
\end{example}

\begin{example}
 Consider the RS $\mathcal J =(M=\{x^2,xy,y^2\}, \{\lambda_{x^2}=\lambda_{y^2}=\lambda_{xy}=\{1\}\} ,  \{\tau_{x^2}=\tau_{y^2}=\tau_{xy}=\TT))$   and the marked set $\mathcal{F}=\{x^2-1, xy,y^2\}$ in  $A[x,y]$.

For  $g=y^3$ and  $l=y$ we have   $g-l=yf_{x^2}-xf_{xy}+yf_{y^2} \in \langle \tau \mathcal F \rangle$ and $\supp(l)\subset  \cN(J)$, but $g$ has only one possible complete reduction
 $g \rightarrow_{\mathcal F \mathcal J} 0\downarrow$ by means of $f_{y^2}$;  therefore,  $g \rightarrow_{\mathcal F \mathcal J}^+ l\downarrow$ does not hold.
Notice that $g=y^3=yf_{y^2} \in \langle \tau \mathcal{F}\rangle$, whereas $l=-yf_{x^2}+xf_{xy} \in \langle \tau \mathcal{F}\rangle\cap \langle \cN(J) \rangle$ is exactly the S-polynomial
$S(f_{x^2},f_{xy})$ (see Remark \ref{SPOL}).
 \end{example}

\section{Confluence and ideal membership}\label{Conflusect}
 The reduction procedure on a polynomial $f$ with respect to a given marked set $\mathcal F$ over a RS  $\mathcal J=(M,\lambda,\tau)$ in general is not unique. 

For instance, we start the reduction choosing a monomial $u$ in $\supp(f)\cap J$ (there could be several)   and a term $m$ in $M$ such that $u\in \cone(m)$ (there could be several).
If $\mathcal J$ is notherian, after a finite number of steps we obtain a reduced form $l$. It is natural to ask whether  $l$ could be independent of  the choices we performed, namely   under which conditions the procedure is confluent. 
 
\begin{definition} \label{defnoetheriana} Let $\mathcal F$ be a marked set over a weakly noetherian RS  $\mathcal J$. 
If for each polynomial  $g$   there is one and only one $l$ s.t.  $g \rightarrow_{\mathcal F \mathcal J}^+ l\downarrow$, then we  call $\rightarrow_{\mathcal F \mathcal J}^+ $   \emph{confluent}.

We  call $\mathcal J$ \emph{confluent} if for each marked set  
    $\mathcal F$  over $\mathcal J$, the reduction procedure 
 $\rightarrow_{\mathcal F \mathcal J}^+$ is confluent.
 \end{definition}

The most significant case of confluent RS is the one presented in the following

\begin{remark}\label{completezza} If $\mathcal J=(M,\lambda,\tau)$ is a weakly noetherian  RS with disjoint cones, then it is  noetherian and confluent.

Since Noetherianity follows by Lemma \ref{implicazionibanali},  we need to show that each marked set   $\mathcal F$ over $\mathcal J$ is confluent.    
If there are $g\in \PP$, $l,l'\in\PP$, $\supp(l)\in\cN(J),$ $\supp(l')\in\cN(J)$,  s.t.  $g \rightarrow_{\mathcal F \mathcal J}^+ l\downarrow$ and   $g \rightarrow_{\mathcal F \mathcal J}^+ l'\downarrow$,  then by  Corollary \ref{sommasemplice}, we would have  $l'-l=(g-l)-(g-l') \in \langle \tau \mathcal F \rangle$, hence $l'-l=0$. 
\end{remark}

\begin{example}\label{Diana}  
The set of all marked sets over the RS 
$$(M, \lambda:=\{\lambda_\alpha\, , \ x^\alpha \in M\}, \tau:=\{\tau_\alpha \, , \ x^\alpha \in M\})$$ 
with $\lambda_\alpha=\emptyset$ for all $x^\alpha \in M$, consists of the single set $\{f_\alpha=x^\alpha :  x^\alpha \in M\}$ namely with  the monomial set $M$ it self.
Therefore $\mathcal J$  is obviously noetherian and confluent. 

If however we assume (as in Buchberger Theory) $\tau_\alpha=\mathcal T$ for all $x^\alpha \in M$,  then the cones of two different monomials in $M$  are not disjoint!  
\end{example}

Of course, a RS  $\mathcal J=(M,\lambda,\tau)$ coherent with a term order and with maximal cones is both noetherian and with non-disjoint cones (unless $\# M=1$).  In this ``natural'' setting confluency is related with ideal membership.
On the other side, Janet (followed by all research in involutiveness) introduced, in the reduction step related with \emph{membership test}, the restriction to disjoint cones thus trivially guaranteeing confluence; the counterpart, clearly, is that one has to transfer to a different procedure the task of granting that   the $A$-vectorspace $\langle \tau \mathcal F\rangle$ generated by the set $\tau \mathcal F$ of all polynomials $x^\gamma f_\alpha$ with $f_\alpha \in \mathcal F$ and $x^\gamma \in \tau_\alpha$ which, in principle is just a subvectorspace of the ideal
$(\mathcal F)$ generated by $\mathcal F$, really coincides with it; Janet approach was, originally via Riquier's \emph{completion}, later, in connection with Cartan test, with   complete linear reduction of sufficiently many vectorspaces $\mathcal F_d$.
\medskip

Let $\mathcal J$ be a weakly noetherian RS.  Even if the cones in $\mathcal J$ are not disjoint, we can  \lq\lq simulate\rq\rq\   this property in the following way.  

Let $\tilde{\tau}=\{\tilde{\tau}_\alpha,  x^\alpha \in M\}$  be s.t. each $\tilde{\tau}_\alpha$ is a subset of $  \tau_\alpha$; in what follows we will denote by $\rightarrow_{\tilde{\tau}\mathcal F \mathcal J}^+  $ the reduction process obtained by  using only polynomials of  $\tilde{\tau}{\mathcal F} :=\{ x^\eta f_\alpha \, \vert \,  f_\alpha \in \mathcal F, \, x^\eta \in  \tilde{\tau}_\alpha \}$. 

\begin{lemma} \label{limitiamo}   Let    $\mathcal J=(M, \lambda ,  \tau )$ be a weakly noetherian RS. Then, there is a list of sets of terms $ \overline{\tau}=\{ \overline{\tau}_\alpha\}_{x^\alpha\in M} $ with  $\overline{\tau}_\alpha\subseteq \tau_\alpha$  s.t. 
\begin{itemize}
\item $\forall  x^\alpha, x^{\alpha'}\in M$, $ x^\alpha \neq x^{\alpha'}$, one has  $x^\alpha \overline{\tau}_\alpha \cap  x^{\alpha'} \overline{\tau}_{\alpha'}=\emptyset$ 
 \item  $\bigcup_{x^\alpha \in M} x^\alpha \overline{\tau}_\alpha =J$
\item  for each marked set $\mathcal F$ on  $\mathcal J$,  the reduction process   $ \rightarrow_{\overline{\tau}\mathcal F \mathcal J}^+ $  is  noetherian.
\end{itemize}
 \end{lemma}
\begin{proof} 
By hypothesis there is a noetherian substructure $\mathcal J'=(M, \lambda, \tau')$ of $\mathcal J$, so  $\tau'_\alpha \subseteq \tau_\alpha$ and  $J=\bigcup_{x^\alpha \in M} x^\alpha {\tau'}_\alpha$.

We can construct the required subsets  $\overline{\tau}_\alpha$ of $\tau'_\alpha$    as follows:  for  each $x^\beta\in J$  we choose randomly one and only one monomial  $x^\alpha\in M$ s.t. $x^\beta=x^\gamma x^\alpha$ with $x^\gamma \in \tau'_\alpha$ and insert $x^\gamma$ in $\overline{\tau}_\alpha$.    

Of course this is all one needs to find subsets $\overline{\tau}_\alpha\subseteq \tau_\alpha$ and  grant that the first two conditions are satisfied; moreover,  noetherianity of $\mathcal J'$ grants noetherianity of $ \rightarrow_{\overline{\tau}\mathcal F \mathcal J}^+$. 
\end{proof}
By a  restriction to disjoint cones we can  now  reinforce point  iii)  of Proposition \ref{ScritturaDistinti1}.
\begin{proposition}\label{sommadiretta1}Let  $\mathcal F$ be a marked set over a weakly noetherian RS  $\mathcal J$,  $\mathcal J'$ be a noetherian substructure (with order $\Prec$) and   $\overline{\tau}$  be as in Lemma \ref{limitiamo}.

Then, $\forall g\in \PP$, there exists  a unique $J$-remainder $l$   s.t.  $g \rightarrow_{\overline{\tau}\mathcal F \mathcal J}^+ l $.   Moreover,  $g-l$ has a representation
 $\sum_{j}c_j x^{\gamma_j}f_{\alpha_j}$ by  $\overline{ \tau} \mathcal F $ with all distinct heads and   $x^{\gamma_j+\alpha_j} \Prec  x^\delta$ for some $x^\delta \in \supp(g)$,  and $l=0$ if and only if  $ g \in \langle {\overline{\tau}}{\mathcal F}\rangle$.
Therefore
\[   \langle {\overline{\tau}}{\mathcal F}\rangle \oplus \langle  {\cN}(J)\rangle = \PP .\]
\end{proposition}
\begin{proof} 
For every polynomial  $g\in \PP$, the $J$-remainder  $l$ exists and is unique  by Lemma \ref{limitiamo}.
Notice that the elements of ${\overline{\tau}}{\mathcal F}   $ have all distinct heads; moreover   $  \rightarrow_{\overline{\tau}\mathcal F \mathcal J}^+ $ is   noetherian  since  $\mathcal J'$ is  noetherian.  We conclude by Corollary \ref{sommasemplice}.
\end{proof}

We can now characterize confluency of marked sets over  weakly noetherian RSs

\begin{theorem} \label{caratterizzazioneconfluenza}
Let $\mathcal F$ be a marked set over a weakly noetherian RS  $\mathcal J$ and let $\mathcal J'$ and    $\overline{\tau} $  be as in Lemma \ref{limitiamo}. The following statements are equivalent:

\begin{enumerate}[i)]
\item \label{caratterizzazioneconfluenza_i} $\rightarrow_{\mathcal F \mathcal J}^+ $ is confluent.
\item \label{caratterizzazioneconfluenza_ii}   $ \langle \tau \mathcal F \rangle  \oplus \langle {\cN}(J)\rangle = \PP$.
\item  \label{caratterizzazioneconfluenza_iii}   $ \langle \tau \mathcal F \rangle  \cap  \langle {\cN}(J)\rangle=0$.
\item  \label{caratterizzazioneconfluenza_iv} $ \langle\tau \mathcal F\rangle =\langle\tau' \mathcal F\rangle = \langle {\overline{\tau}}{\mathcal F}\rangle $.
\item \label{caratterizzazioneconfluenza_v}  for each  $ x^{\eta} f_{\alpha} \in \tau\mathcal F \setminus  {\overline{\tau}}{\mathcal F}$ it holds $ x^{\eta} f_{\alpha} { \rightarrow_{\overline{\tau}\mathcal F \mathcal J}^+ } \ 0\downarrow$.  
\item  \label{caratterizzazioneconfluenza_vi} for each  $  x^{\eta} f_{\alpha}\in \tau\mathcal F $,  for each reduction $ x^{\eta} f_{\alpha}  \rightarrow_{\mathcal F \mathcal J'}^+ l\downarrow$ it holds $l=0$. 
\item \label{caratterizzazioneconfluenza_vii}  for each  $  x^{\eta} f_{\alpha}\in \tau\mathcal F $, $   x^{\eta'} f_{\alpha'} \in \tau'\mathcal F $ with $x^{\eta+\alpha}=x^{\eta'+\alpha'}$ it holds  $ x^{\eta} f_{\alpha}- x^{\eta'} f_{\alpha'} \rightarrow_{\mathcal F \mathcal J'}^+ 0\downarrow$.  
 \item  \label{caratterizzazioneconfluenza_viii}    for each $  x^{\eta} f_{\alpha}  \in \tau\mathcal F $, for each reduction $ x^{\eta} f_{\alpha} \rightarrow_{\mathcal F \mathcal J}^+ l\downarrow$ it holds $l=0$.  
 \item  \label{caratterizzazioneconfluenza_ix}  for each $  x^{\eta} f_{\alpha}, x^{\eta'} f_{\alpha'} \in \tau\mathcal F $ with $x^{\eta+\alpha}=x^{\eta'+\alpha'}$, for each reduction $ x^{\eta} f_{\alpha}- x^{\eta'} f_{\alpha'} \rightarrow_{\mathcal F \mathcal J}^+ l\downarrow$ it holds $l=0$.  
\end{enumerate}
\end{theorem}

\begin{proof} 
ii) $\Leftrightarrow $ iii): the assertion trivially follows from Corollary \ref{sommasemplice}.
 
iii) $\Rightarrow $ i): notice that if  $ g \rightarrow_{\mathcal F \mathcal J}^+ l\downarrow$ and $g \rightarrow_{\mathcal F \mathcal J}^+ l'\downarrow$, then the difference $l-l'$ belongs to $ \langle \tau \mathcal F \rangle  \cap  \langle {\cN}(J)\rangle$, so  $l-l'=0$.

ii) $\Leftrightarrow $ iv) $\Leftrightarrow $ v): follow from   Proposition  \ref{sommadiretta1}   and from the fact that by construction  $  \tau \mathcal F \supseteq\tau' \mathcal F \supseteq   \overline \tau {\mathcal F}  $.  

 viii) $\Rightarrow $ vi)  $\Rightarrow $ v)  are trivial, since the reductions  by  $ \rightarrow_{\overline{\tau}\mathcal F \mathcal J}^+$ are particular cases of the ones by $\rightarrow_{\mathcal F \mathcal J'}^+ $, that are particular cases of the ones by  $ \rightarrow_{\mathcal F \mathcal J}^+ $. 
 Notice that    $\mathcal F$ is weakly noetherian,   so each polynomial has at least a total reduction.  
 
i) $\Rightarrow $ viii)  is again obvious; indeed every polynomial $x^\eta f_\alpha\in \tau\mathcal F  $ has at least the complete reduction  $x^\eta f_\alpha  \rightarrow_{\mathcal F \mathcal J} x^\eta f_\alpha-x^\eta f_\alpha=0 \downarrow$.

\medskip

As a consequence of what  proved so far,  the conditions i),  ii),  iii), iv), v), vi), viii) are equivalent. 

\medskip

 iii) $\Rightarrow$   ix): it is sufficient to observe that in the hypotheses  ix) the polynomial  $ x^{\eta} f_{\alpha}- x^{\eta'} f_{\alpha'}$ belongs to $\tau \mathcal F$   and    $l $ is in the intersection $ \langle \tau \mathcal F \rangle  \cap  \langle {\cN}(J)\rangle$.
 
ix) $\Rightarrow$  vii) directly follows by the same argument used to prove \lq \lq viii) $\Rightarrow $ vi)  $\Rightarrow $ v) \rq\rq.

\medskip

 We finally prove   vii)  $\Rightarrow $ iv). By     Proposition \ref{ScritturaDistinti1} i),    condition vii) implies $\langle \tau \mathcal F \rangle \subseteq \langle \tau' \mathcal F \rangle$. Then, it is sufficient to prove that $\langle \tau' \mathcal F \rangle \subseteq \langle \overline \tau \mathcal F \rangle$, the opposite inclusions being obvious.    

 Assume by contradiction that  the set $ \tau' \mathcal F \setminus  \langle \overline \tau \mathcal F \rangle $  is not empty and choose in it an element  $ x^{\eta} f_{\alpha}$ such that is  minimal $x^{\eta +\alpha}$, w.r.t. the  ordering $\Prec$  associated to  $\mathcal J'$. Let, moreover, $   x^{\eta'} f_{\alpha'}$ the only element in  $  \overline \tau \mathcal F  $ such $x^{\eta+\alpha}=x^{\eta'+\alpha'}$: we may apply vii) to these two elements (as  $ x^{\eta} f_{\alpha} \in \tau' \mathcal F \subseteq \tau \mathcal F$  and $x^{\eta'} f_{\alpha'}\in  \overline \tau \mathcal F \subseteq  \tau' \mathcal F  $) finding a reduction  $ x^{\eta} f_{\alpha}- x^{\eta'} f_{\alpha'}\rightarrow_{\mathcal F \mathcal J'}^+ 0\downarrow$.  

We observe that  every term $x^\gamma \in \supp( x^{\eta} f_{\alpha}- x^{\eta'} f_{\alpha'})$ is either in $\supp(x^{\eta} f_{\alpha}-x^{\eta+\alpha})$ or in  $\supp(x^{\eta'} f_{\alpha'}-x^{\eta'+\alpha'})$. In both cases,   $ x^\gamma \Prec   x^{\eta+\alpha}=x^{\eta'+\alpha'}$.

Then, by   Corollary  \ref{sommasemplice}, the polynomial       $ x^{\eta} f_{\alpha}- x^{\eta'} f_{\alpha'}$ has a $x^{\eta+\alpha}-\SLR$   in $\tau'\mathcal F$ of the type $\sum_{i=1}^r c_i x^{\gamma_i}f_{\alpha_i}$. By the  minimality of $x^{\eta+\alpha}$ we deduce that  the summands $x^{\gamma_i}f_{\alpha_i}$ belong to $ \langle \overline \tau \mathcal F \rangle $, hence also 
$ x^{\eta} f_{\alpha}- x^{\eta'} f_{\alpha'}$ does. This is the wanted contradiction, as $ x^{\eta} f_{\alpha} \notin \langle \overline \tau \mathcal F \rangle $ and $ x^{\eta'} f_{\alpha'} \in \overline \tau \mathcal F  $.
\end{proof} 
Now we assume that  $\mathcal J$ is a weakly noetherian RS and    see which conditions have to be satisfied by a marked set
$\mathcal F$   over  $\mathcal J$ in order that the rewriting rule  $ \rightarrow_{\mathcal F \mathcal J}^+ $ give a criterion which is equivalent to the belonging to the ideal generated by  $\mathcal F$, i.e. in order that 
$$    g\equiv g' \bmod{(\mathcal F)} \Longleftrightarrow  \forall
g \rightarrow_{\mathcal F \mathcal J}^+ l\downarrow  \hbox{  and  } \forall    
g' \rightarrow_{\mathcal F \mathcal J}^+ l'\downarrow, \hbox{ it holds } l=l'$$
is satisfied.

We can observe that in order to apply the test implied by $\Leftarrow$ for deciding ideal equivalence (and in particular,
 ideal membership ) we must require that $\mathcal J$ is weakly noetherian; 
 indeed, if there is a polynomial $g$ without complete reductions, the reduction 
cannot allow us to establish whether $g$ belongs to $ (\mathcal F)$\footnote{This is the ``flaw'' of Hironaka Theory (see \cite{7var}).}

The ideal membership can be reformulated through  the notion of marked bases (Definition \ref{BaseMark}), which constitutes a central point for the whole theory.

\begin{theorem}\label{sommadiretta}
Let $\mathcal J=(M,\lambda,\tau)$ be a weakly noetherian RS and let  $\mathcal J'$  and $\overline{\tau}$ be as in Lemma \ref{limitiamo}. 
Moreover, let   $\mathcal F$ be a marked set over $\mathcal J$.

If $\mathcal F$ is a marked basis, then $ \rightarrow_{\mathcal F \mathcal J}^+$ is confluent. 

On the other hand, if we suppose that    $\rightarrow_{\mathcal F \mathcal J}^+ $ is confluent, then $\mathcal F$ is a marked basis if and only if one of the following equivalent conditions holds:
\begin{enumerate}[i)]
\item \label{sommadiretta_i}  $(\mathcal F)=\langle \overline{\tau} \mathcal F \rangle,  $
\item \label{sommadiretta_ii}  $(\mathcal F)=\langle \tau' \mathcal F \rangle,  $
\item \label{sommadiretta_iii} $(\mathcal F)=\langle  {\tau}{\mathcal F}\rangle, $
\item \label{sommadiretta_iv} for each  $x^\alpha \in M$ and each $x^\gamma \notin \overline{\tau}_\alpha$   it holds  $x^\gamma f_\alpha   \rightarrow_{\overline{\tau}\mathcal F \mathcal J}^+ 0\downarrow$ ,
\item \label{sommadiretta_v} for each $x^\alpha \in M$ and $x^\gamma \notin \tau'_\alpha$   it holds $x^\gamma f_\alpha  \rightarrow_{\mathcal F \mathcal J'}^+ 0\downarrow $,
\item \label{sommadiretta_vi}  for each $x^\alpha \in M$ and  $x^\gamma \notin \tau_\alpha$   it holds  $x^\gamma f_\alpha \rightarrow_{\mathcal F \mathcal J}^+ 0\downarrow$.
\end{enumerate}
\end{theorem}
\begin{proof}    If $\mathcal F$ is a marked basis,   we have   $  \langle \tau  \mathcal F\rangle  \cap  \langle {\cN}(J)\rangle\subseteq ( \mathcal F)  \cap \langle {\cN}(J)\rangle {=0}$; then $\rightarrow_{\mathcal F \mathcal J}^+$ is confluent by Theorem \ref{caratterizzazioneconfluenza} iii) $\Rightarrow $ i).  

Now,  assume that $ \rightarrow_{\mathcal F \mathcal J}^+ $ is confluent.  

The conditions  {\it \ref{sommadiretta_i})}, {\it \ref{sommadiretta_ii})}, {\it \ref{sommadiretta_iii})} are equivalent by  Theorem \ref{caratterizzazioneconfluenza} and the conditions  {\it \ref{sommadiretta_iv})}, {\it \ref{sommadiretta_v})}, {\it \ref{sommadiretta_vi})} are equivalent since   the confluence of $ \rightarrow_{\mathcal F \mathcal J}^+ $ grants also the confluence of  $ \rightarrow_{\mathcal F \mathcal J'}^+$ and $  \rightarrow_{\overline{\tau}\mathcal F \mathcal J}^+$.   Notice that in each of the three conditions  {\it \ref{sommadiretta_iv})}, {\it \ref{sommadiretta_v})}, {\it \ref{sommadiretta_vi})} the restriction on $x^\gamma$    could  be omitted; indeed, if  for instance $x^\gamma \in \tau_\alpha$, then   by a single step of reduction on $x^{\gamma+\alpha}$ we obtain   $x^\gamma f_\alpha \rightarrow_{\mathcal F \mathcal J}   x^\gamma f_\alpha -x^\gamma f_\alpha = 0 \downarrow$.
Finally, the equivalence between     {\it \ref{sommadiretta_i})}   and  {\it \ref{sommadiretta_iv})}  is consequence of Proposition \ref{sommadiretta1} and of the above remark about  $x^\gamma$.

 We conclude observing that    $\langle \overline{\tau}\mathcal F\rangle \subseteq (\mathcal F)$, so by Proposition \ref{sommadiretta1}, $\mathcal F$ is a marked basis if and only if    {\it \ref{sommadiretta_i})}  holds. 
\end{proof}

\begin{remark}\label{SPOL}
We can reformulate the characterizations of confluence of  Theorem \ref{caratterizzazioneconfluenza}  and of marked bases of  Theorem \ref {sommadiretta} using the reduction w.r.t. polynomials of the form $ x^{\eta} f_{\alpha}- x^{\eta'} f_{\alpha'}$ with  $x^{\eta+\alpha}=x^{\eta'+\alpha'}$.
Notice, anyway, that they are not only the   S-polynomials $S( f_{\alpha'},   f_{\alpha}):= x^{\eta} f_{\alpha}- x^{\eta'} f_{\alpha'}$, with     $x^{\eta+\alpha}=x^{\eta'+\alpha'}=\lcm(x^\alpha, x^{\alpha'})$, but  a priori   also all their (infinite) multiples.

In  Appendix we will prove that for every weakly noetherian RS there exists a finite set of  controls using reductions that  are  sufficient to ensure that a marked set is a marked basis.  However, this result is non-constructive. In particular, we do not have neither a proof nor a counter-example to the reasonable conjecture that the set of $S$-polynomials could be sufficient to this purpose.
 
For this reason, for practical purposes, it is necessary to consider RSs  with particular properties, allowing to execute those verifications  with a  known, finite (possibly small)   set  of reductions. We will examine  two  sufficiently general cases in Sections \ref{sez:disgiunti} and \ref{sez:ordered}; in both of them the set of controls corresponds to the set of $S$-polynomials or a subset of it. 
\end{remark}

\section{Maximal and disjoint cones: criteria for marked bases}\label{sez:disgiunti}

In the usual reduction procedure w.r.t. a set of marked polynomials, one admits to rewrite any multiple of 
 $x^\alpha$ with the marked polynomial $f_\alpha$  whose head is  $x^\alpha$. 
 In our language, every  term in  $\TT$ is considered as  multiplicative  for each $x^\alpha \in M$: these are the structures we call \emph{with maximal cones}.

 If such a RS  $\mathcal J=(M,\lambda,\tau)$ is noetherian we already remarked that it must be necessarily coherent with a term order by Theorem  \ref{REEVESTURMFELS}. Then the marked bases over $\mathcal J$ are Gr\"obner bases. Moreover, for a set $\mathcal F$ marked over $\mathcal J$ the fact of being a marked basis and the confluency of $ \rightarrow_{\mathcal F \mathcal J}^+ $  are equivalent,  since $(\mathcal F)$ and $\langle \tau \mathcal F \rangle$ coincide by construction.

It is a well known fact that in the Gr\"obner case, in order to check whether a marked set is a marked basis  ({\em id est} a Gr\"obner basis)
  it is sufficient to perform a finite number of controls   which can be deduced by the given data, namely 
Buchberger test/completion result states that a basis (in our language: a marked set) $\mathcal{F}$ is Gr\"obner  (in our language: a marked basis) if and only if each element in the set of all S-polynomials
$$\left\{S( f_{\alpha'},   f_{\alpha}):= \frac{\lcm(x^\alpha, x^{\alpha'})}{x^\alpha} 
 f_{\alpha}- \frac{\lcm(x^\alpha, x^{\alpha'})}{x^{\alpha'}}  f_{\alpha'} : x^\alpha, x^{\alpha'}\in M\right\}$$
between two elements of $\mathcal{F}$, reduces to 0.

Thus we do not need to check  any of their multiples.  
 
 Being a well known theory, we do not treat it in the usual way, but we change our point of view.

As underlined in Remark \ref{codefinite},  the concept of marked basis depends only   on     $\mathcal F$ and it does not depend   on   the RS over which we consider it   as    a marked set.   In order to characterize the marked bases over  $\mathcal J$, a substructure   $\mathcal J'$ of $\mathcal J$ having disjoint cones  (if it exists) could be useful; when,  as in Gr\"obner theory,  $\mathcal J$ has maximal cones,  such a substructure exists.
We propose here one of the possible ways to   construct it. 

\begin{lemma} \label{sottoconidisgiunti} If $\mathcal J=(M,\lambda, \tau)$ is a RS with maximal cones, then there is a substructure $\mathcal J'=(M, \lambda, \tau')$ with disjoint cones.   
\end{lemma}
\begin{proof} Consider the set $M=\{x^{\alpha_1}, \dots, x^{\alpha_s}\}$ and suppose that its terms are ordered in such a way that none of the   $x^{\alpha_i} $ is multiple of a term with an index $<  i$. 

 First of all, set  $\tau'_{\alpha_1}:=\TT$ then  $\tau'_{\alpha_2}x^{\alpha_2}: =x^{\alpha_2}  \TT\setminus x^{\alpha_1}\TT$.  
 Notice that  $\tau'_{\alpha_2}$ is an order ideal (in particular $1\in \tau'_{\alpha_2}$) since $x^{\alpha_1}\nmid x^{\alpha_2}$.

By induction, after determining the multiplicative sets of the first $r$ terms of $M$,  set 
 $x^{\alpha_{r+1}}\tau'_{\alpha_{r+1}}:=x^{\alpha_{r+1}}  \TT \setminus \bigcup_{i=1}^r x^{\alpha_i}\tau'_{\alpha_{i}}$.   
\end{proof}

In the Gr\"obner case,  $\mathcal J$ has maximal cones and is noetherian, i.e.   coherent with a term order $\prec$ (Theorem \ref{REEVESTURMFELS}), then  every   substructure $\mathcal J'$ of $\mathcal J$ with disjoint cones is   noetherian,  coherent with $\prec$,  and confluent. 
  
  We prove now that  the well known criteria to check if a markes set is a basis that appear in the Gr\"obner  theory are sufficient also in a more general setting that only assume a proper subset of the above conditions.

In the last part of this section, we will study the properties of noetherian RSs  with disjoint cones, for which  the  following condition on the well-founded order $\Prec$  holds:
\begin{equation} \label{piuordinato} \forall x^\delta, x^{\delta'}, x^\varepsilon \in \TT: \  x^\delta \Succ x^{\delta'} \Rightarrow x^{\delta + \varepsilon} \Succ x^{\delta'+\varepsilon}\Succeq   x^{\varepsilon}.\end{equation}

This condition clearly holds  if $\mathcal J'$ is coherent with a term order $\prec$ and $\Prec$ is exactly this term ordering.

\begin{proposition}\label{comeBuch1}  Let $\mathcal J'$ be a noetherian RS with disjoint cones and suppose also that   \eqref{piuordinato} holds.

If $\mathcal F$  is a marked set over $\mathcal J'$ and $x^\beta $ is a term,  the following are equivalent:
\begin{enumerate}[i)]
\item \label{comeBuch1_i}  for each $x^\alpha \in  M$,   $x^\eta \notin \tau_\alpha$  s.t.  $x^{\eta+\alpha} \Prec  x^\beta$,  it holds  $  x^{\eta} f_{\alpha}\in \langle \tau'\mathcal F\rangle $ 
\item \label{comeBuch1_ii}  for each  $x^\alpha \in  M$,   $x^\eta \notin \tau_\alpha$  s.t.  $x^{\eta+\alpha} \Prec  x^\beta $,   $  x^{\eta} f_{\alpha}$ 
has a  $x^\beta-\SLR$.
 \item \label{comeBuch1_iii} for each S-polynomial  $S( f_{\alpha'},   f_{\alpha})$  s.t.   
$${\lcm(x^\alpha, x^{\alpha'})} \in \cone(x^{\alpha'})\setminus\cone(x^\alpha) \mbox{\ and } 
 \lcm(x^\alpha, x^{\alpha'})\Prec  x^\beta,$$   $S(f_{\alpha'}, f_{\alpha})$ has a $x^\beta-\SLR$.   
\item  \label{comeBuch1_iv}  in the same hypotheses as iii) it holds 
$S(f_{\alpha'}, f_{\alpha}) \rightarrow_{\mathcal F \mathcal J}^+ 0\downarrow.$
\end{enumerate} 
\end{proposition}
 
\begin{proof}  First of all we observe that in our hypotheses if $x^{\eta+\alpha} \Prec  x^\beta$ then also $x^{\delta} \Prec  x^\beta$ for every term $x^\delta \in \supp(x^\eta f_\alpha)$.

{\it \ref{comeBuch1_i})}    $\Leftrightarrow $  {\it \ref{comeBuch1_ii})} follows by Corollary \ref{sommasemplice}.

 {\it \ref{comeBuch1_ii})}    $\Rightarrow $  {\it \ref{comeBuch1_iii})}  Consider an element $x^\eta f_\alpha$ satisfying the conditions of   {\it \ref{comeBuch1_ii})}.
 Since   {\it \ref{comeBuch1_i})} holds, it has a $x^\beta-\SLR$; summing   $- x^{\eta'} f_{\alpha'}$ we get a   $x^\beta-\SLR$  of  $ x^{\eta} f_{\alpha}- x^{\eta'} f_{\alpha'}$. 

{\it \ref{comeBuch1_iii})}    $\Leftrightarrow $  {\it \ref{comeBuch1_iv})}  comes trivially from Corollary \ref{sommasemplice}.

 {\it \ref{comeBuch1_iii})} $+$  {\it \ref{comeBuch1_iv})}  $\Rightarrow $  {\it \ref{comeBuch1_ii})}  Suppose by contradiction that the assertion is false and that $x^\beta$ is a  term with $x^\beta$ minimal among the ones not satisfying the condition.  Consider $x^\eta f_\alpha$ with  $x^{\eta+\alpha}\Prec x^\beta$.  Notice that, by hypothesis,  the assertion is true in particular for   $x^{\beta'}:=x^{\eta+\alpha}$.

The assertion in {\it \ref{comeBuch1_ii})}  would immediately follow by {\it \ref{comeBuch1_iii})} if for the only $x^{\eta'} f_{\alpha'} $ s.t.  $x^{\eta+\alpha}=x^{\eta'+\alpha'}\in \cone(x^{\alpha'})$ one has $\lcm(x^\alpha, x^{\alpha'})=x^{\eta+\alpha} $.  So we must have  $\lcm(x^\alpha, x^{\alpha'})=x^{\varepsilon+\alpha} =x^{\varepsilon'+\alpha'}$ with $x^\varepsilon $ proper divisor of $x^\eta$.  We can then apply   {\it \ref{comeBuch1_iv})} getting  $x^\varepsilon f_\alpha-x^{\varepsilon'}f_{\alpha'}=S(f_{\alpha'},f_\alpha) \rightarrow_{\mathcal F \mathcal J}^+ 0\downarrow$.  Notice that by   \eqref{piuordinato} for each term $x^\gamma$ in the support of $x^\varepsilon f_\alpha$ and of $x^{\varepsilon'}f_{\alpha'}$ it holds $x^\gamma \Prec x^{\varepsilon+\alpha}$.  By Corollary \ref{sommasemplice}, we have then  $x^\varepsilon f_\alpha-x^{\varepsilon'}f_{\alpha'}$ has a $x^{\varepsilon+\alpha}-\SLR$, i.e.  $x^\varepsilon f_\alpha-x^{\varepsilon'}f_{\alpha'}=\sum c_ix^{\gamma_i}f_{\alpha_i}$  with $x^{\gamma_i+\alpha_i}\Prec x^{\varepsilon+\alpha}$.  Multiply this representation by $x^{\eta-\varepsilon}$. For each summand  $x^{\eta-\varepsilon+\gamma_i}f_{\alpha_i}$ it holds $x^{\eta-\varepsilon+\gamma_i+\alpha_i}\Prec x^{\eta-\varepsilon+\varepsilon+\alpha}=x^{\eta+\alpha}$.  By the assumption on the truth of our assertion with   $x^{\beta'}=x^{\eta+\alpha}$, each polynomial $x^{\eta-\varepsilon+\gamma_i}f_{\alpha_i}$  has a $x^{\eta+\alpha}-\SLR$.  We then get a  $x^{\eta+\alpha}-\SLR$  of  $x^\eta f_{\alpha}-x^{\eta'}f_{\alpha'}$ so, summing to the two members  $x^{\eta'}f_{\alpha'}$ we get a  $x^{\eta+\alpha}-\LR$  of  $x^{\eta}f_{\alpha}$ since  $x^{\eta'+\alpha'}\in \tau_{\alpha'}$ . We conclude noticing that by hypothesis  $x^{\eta+\alpha}\Prec x^\beta$.
    
\end{proof}
By the previous results and by Theorems \ref{sommadiretta}  and \ref{caratterizzazioneconfluenza} follows

\begin{corollary}\label{comeBuch}  Let $\mathcal J'$ be a noetherian RS with disjoint cones and order $\Prec$. Suppose  that \eqref{piuordinato} holds. 
Then for a marked set $\mathcal F$ over $\mathcal J'$  the following are equivalent:
\begin{enumerate}[i)]
\item $\mathcal F$ is a marked basis
 \item $\forall  x^{\alpha},  x^{\alpha'} \in M$  s.t.  $\lcm(x^\alpha, x^{\alpha'}) \in \cone(x^{\alpha'})$  it holds  $ S( f_{\alpha'},  f_{\alpha}) \rightarrow_{\mathcal F \mathcal J'}^+ 0\downarrow$
 \item $\forall  x^{\alpha},  x^{\alpha'} \in M$  s.t.  $x^{\gamma+\alpha}=\lcm(x^\alpha, x^{\alpha'}) \in \cone(x^{\alpha'})$  it holds  $ x^\gamma f_{\alpha} \rightarrow_{\mathcal F \mathcal J'}^+ 0\downarrow$.    
\end{enumerate}
\end{corollary}

For such RSs we can improve the characterization of marked bases given in Corollary   \ref{comeBuch} similarly to what done for  Gr\"obner bases.  
We can verify that also in this context some of the known simplifications hold.

The ``strategy'' presented here  exploits a substructure of $\mathcal{J}$ with disjoint cones. Such a structure is inspired by (and generalizes) Gebauer-Moeller's \emph{Staggered linear bases}.

\section{Criteria}\label{Criteriasec}

Throughout this section, for notation simplicity, we will assume that the finite set $M$ is  enumerated as
$\{x^{\alpha_1},\ldots,x^{\alpha_s}\}$ and we will relabel each element $f_{\alpha_i}$ in the related marked set  
$$\mathcal{F}=\{f_\alpha\}_{x^\alpha\in M} = \{f_{\alpha_i}, 1\leq i \leq s\}$$
as $f_i := f_{\alpha_i},1\leq i \leq s$.

\def\then{\;\Longrightarrow\;}
We will further assume to have performed the construction outlines in Lemma~\ref{sottoconidisgiunti}; in particular we have 
$$\tau'_{\alpha_1} = \TT \mbox{ and } x^{\alpha_{r+1}}\tau'_{\alpha_{r+1}}:=x^{\alpha_{r+1}}  \TT \setminus \bigcup_{i=1}^r x^{\alpha_i}\tau'_{\alpha_{i}} \mbox{ for all } i;$$
Further we will assume that the elements of $M$ are ordered so that
\begin{equation}\label{EqTeo}
x^{\alpha_i}\mid x^{\alpha_j}\then i < j.
\end{equation} 
 
We moreover denote
\begin{itemize}
\item for each $i, 1\leq i \leq s, {\bf T}(i) :=x^{\alpha_i}$,  
\item for each  $i,j, 1\leq i \leq s$, $${\bf T}(i,j):=\lcm({\bf T}(f_i),{\bf T}(f_j)) =
\lcm(x^{\alpha_i},x^{\alpha_j})$$
and 
\item $S(i,j) := S(f_i,f_j) = {\frac{{\bf T}(i,j)}{{\bf T}(j)}} f_j - {\frac{{\bf T}(i,j)}{{\bf T}(i)}} f_i $;
\item for each $i,j,k : 1 \leq i , j , k \leq s$, $${\bf T}(i,j,k):=\lcm({\bf T}(f_i),{\bf T}(f_j),{\bf T}(f_k)) =
\lcm(x^{\alpha_i},x^{\alpha_j},x^{\alpha_k}).$$
\end{itemize}
\begin{lemma}[M\"oller] \cite{M}\label{SyzRel}
 For each $i,j,k : 1 \leq i , j , k \leq s$ it holds
$$\frac{{\bf T}(i,j,k)}{{\bf T}(i,k)} S(i,k) - \frac{{\bf T}(i,j,k)}{{\bf T}(i,j)} S(i,j) + 
\frac{{\bf T}(i,j,k)}{{\bf T}(k,j)} S(k,j) 
= 0.$$ 
\end{lemma}

Buchberger test/completion result states  that a basis (in our language: a marked set) $\mathcal{F}$ is Gr\"obner  (in our language: a marked basis) 
if and only if each 
S-polynomial $S(i,j), i,j, 1\leq i \leq s,$ between two elements of $\mathcal{F}$, reduces to 0 and gave two criteria \cite{BCrit} to detect S-pairs which are ``useless'' in the sense that theoretical results prove that they reduce to 0, thus making useless the normal form computation.
The First Criterion (Propostion~\ref{comeBuchMigli1}) is based on a direct reformulation of trivial syzygies, the Second is a direct application Lemma~\ref{SyzRel}.

We remark that the test/completion result given by  Proposition~\ref{comeBuch1}.iv)  allow to remove many useless S-pairs.

In fact, an S-polynomial $S(i, j)$ is not to be tested, and thus considered ``useless'', if 
${\bf T}(i,j)\notin\cone({\bf T}(i))\cup\cone({\bf T}(j))$.

\begin{example}\label{XaX1} Let us consider
$M:=\{x^{\alpha_i} : 1\leq i \leq 3\}$ with
\begin{itemize}
\item $x^{\alpha_1}={\bf T}(1)=xy, \tau_{\alpha_1}=\TT$, 
\item $x^{\alpha_2}={\bf T}(2)=y^2, \tau_{\alpha_2}=\{y^i :i\in{\mathbb N}\}$, 
\item $x^{\alpha_3}={\bf T}(3)=x^2, \tau_{\alpha_3}=\{x^i :i\in{\mathbb N}\}$
\end{itemize}
and remark that
$$S(2,3)=yS(1,3)-xS(1,2).$$

Note that 
$${\bf T}(2,3)=x^2y^2\notin\cone({\bf T}(2))\cup\cone({\bf T}(3))=\{y^{i+2} :i\in{\mathbb N}\}\cup\{x^{i+2} :i\in{\mathbb N}\}$$ while
$$\frac{{\bf T}(1,2)}{{\bf T}(1)}=y\in\tau_{\alpha_1}=\TT\ni y
\frac{{\bf T}(1,3)}{{\bf T}(1)}$$
so we detect the ``useless'' pair $S(2,3)$.
\end{example}

Naturally, we can prove in our setting  Buchberger Second Criterion; we also can prove Buchberger First Criterion

\begin{proposition}\label{comeBuchMigli1} \cite{BCrit} (Buchberger First Criterion)       Under the hypotheses of  Corollary \ref{comeBuch}  for $\mathcal F$ being a marked basis it is not necessary to check whether the 

 $S$-polynomials     $S(f_{\alpha'},f_\alpha)$ s.t. $\lcm(x^\alpha, x^{\alpha'})=x^{\alpha+\alpha'}$ reduce to $0$.
\end{proposition}
\begin{proof}   Suppose $\lcm(x^\alpha, x^{\alpha'})=x^{\alpha+\alpha'}$.  Apply Proposition \ref{comeBuch1}  choosing  $x^\beta=x^{\alpha+\alpha'}$. If some of the requested controls is negative, $\mathcal F$ is not a marked basis and we can conclude it without using $S(f_{\alpha'},f_\alpha)$.  Otherwise,    all the polynomials  $x^\epsilon  f_{\alpha''}$ with   $x^{\epsilon +\alpha''}\Prec x^{\alpha'+\alpha}$ belong to  $\langle \tau \mathcal F \rangle $.

Denoted $f_\alpha=x^\alpha -g_\alpha$ and $f_{\alpha'}=x^{\alpha'} -g_{\alpha'}$, it holds $x^{\alpha'} f_{\alpha}- x^{\alpha} f_{\alpha'}= g_{\alpha'} f_{\alpha}- g_{\alpha} f_{\alpha'}$. 
By definition of ordered RS, all the terms $x^\delta$ in the support of  $g_\alpha$ are s.t. $x^\delta\Prec x^\alpha$, so by \eqref{piuordinato} we have  $x^{\delta+\alpha'}\Prec x^{\alpha+\alpha'}$. Then $g_{\alpha} f_{\alpha'}\in  \langle \tau \mathcal F \rangle $. Similarly we get   $g_{\alpha} f_{\alpha'}\in \langle \tau \mathcal F \rangle $ and we conclude that their difference $S(f_{\alpha'},f_\alpha)$ is in $\langle \tau \mathcal F \rangle $.
\end{proof}

Differently from Gr\"obner bases, it is not always true that the S-polynomial of two polynomials with coprime heads reduces to $0$.

 \begin{example}\label{NoCoprime1}
Consider the RS with $M=\{x,y,xz\}\subset \PP=A[x,y,z]$,   $\tau_{x}=\TT[x,y]$, $\tau_{y}= \tau_{xz}=\TT $. 
Take 
 $\mathcal{F}=\{f_{x}=x, f_{y}=y-z,f_{xz}=xz-z^2 \}$. We will have then $yf_x, xf_y\in\tau  \mathcal F $, but the only reduction of the S-polynomial $S(f_{y }, f_{x})=yf_x-xf_y=xz  \rightarrow_{\mathcal F \mathcal J} z^2 \downarrow $ (by means of $f_{xz}$) does not produce $0$.
 The point, of course, is that (\ref{piuordinato}) is not satisfied
\end{example}

\begin{proposition}  \label{comeBuchMigli2}\cite{BCrit}  (Buchberger  Second Criterion)  Under the hypotheses of  Corollary \ref{comeBuch},  for $\mathcal F$ being a marked basis it is not necessary to control that 
 $S(f_{\alpha'}, f_{\alpha''})$  reduces to $0$   if we already checked   $S(f_{\alpha'},f_\alpha)$ and $S(f_{\alpha''},f_\alpha)$, and  $x^{\alpha}\mid \lcm(x^{\alpha'}, x^{\alpha''})$.

\end{proposition}
\begin{proof}  By hypothesis and Lemma~\ref{SyzRel} $S(f_{\alpha'}, f_{\alpha''})=x^{\varepsilon'}S(f_{\alpha'},f_\alpha)-x^{\varepsilon''}S(f_{\alpha''},f_\alpha)$  for some $x^{\varepsilon'}, x^{\varepsilon''}\in \TT$.  
Apply  Proposition \ref{comeBuch1}  choosing  $x^\beta=\lcm(x^{\alpha'}, x^{\alpha''})$. If some of the requested controls is negative, $\mathcal F$ is not a marked basis and we can conclude it without using   $S(f_{\alpha'}, f_{\alpha''})$.
Otherwise,   we know that all the polynomials  
$x^\epsilon  f_{\gamma}$  with  $x^{\epsilon +\gamma}\Prec x^{\alpha'+\alpha}$  are in $\langle \tau \mathcal F \rangle $. 
By hypothesis we also know that  $S(f_{\alpha'},f_\alpha) \in \langle \tau \mathcal F\rangle$; so we can write it by a 
$\lcm(x^{\alpha+\alpha'})$-\SLR\ since for each term $x^\delta$  in the support of $S(f_{\alpha'},f_\alpha)$ one has 
$x^\delta \Prec \lcm(x^{\alpha},x^{\alpha'})$. Then, multiplying the summands  $x^{\eta_i}f_{\alpha_i}$ of this representation by   $x^{\varepsilon'}$,  we get polynomials $ x^{\varepsilon'+\eta_i}f_{\alpha_i}$   belonging to $\langle \tau'\mathcal F\rangle$ since     $x^{\varepsilon'+\eta_i+\alpha_i}\Prec x^{\varepsilon'}\lcm(x^\alpha,x^{\alpha'})=\lcm(x^{\alpha'},x^{\alpha''})$.

  Then $x^{\varepsilon'}S(f_{\alpha'},f_\alpha)$ is in  $\langle \tau'\mathcal F\rangle$. Similarly we can obtain that $x^{\varepsilon''}S(f_{\alpha''},f_\alpha)$ is in $\langle \tau'\mathcal F\rangle$ and we conclude.

\end{proof}

Let us now enumerate the set of all S-pairs by a well-founded order  $\prec$ which preserves  divisibility:
\def\then{\;\Longrightarrow\;}

\begin{equation}\label{Eqprec}
{\bf T}(i_1,j_1) \mid {\bf T}(i_2,j_2) \neq {\bf T}(i_1,j_1)\then (i_1,j_1) \prec (i_2,j_2)
\end{equation}
\begin{corollary}[Buchberger]\cite{BCrit}\cite[II.Lemma~25.1.3]{SPES}\label{GMCriCor}
Let $$\mathfrak{B}  \subset \{\{i,j\}, 1\leq i < j \leq s\}$$ be such that
for each $\{i,j\}, 1\leq i < j \leq s$, either
\begin{itemize}
\item ${\bf T}(i,j) =  {\bf T}(i) {\bf T}(j)$ or
\item there is $k, 1\leq k \leq s$ such that
\begin{itemize}
\item ${\bf T}(k)\mid {\bf T}(i,j)$ and
\item $\{i,k\}\prec\{i,j\}$
\item $\{k,j\}\prec\{i,j\}$.
\end{itemize}
\end{itemize} 

Then under the hypotheses of  Corollary \ref{comeBuch}  for $\mathcal F$ being a marked basis 
it is sufficient to check whether the S-polynomials belonging to $\{\{i,j\}, 1\leq i < j \leq s\}\setminus\mathfrak{B}$ for $\mathcal{F}$ reduce to 0.
\end{corollary}

\begin{proof} The proof is performed by induction according $\prec$:
for each $i,j, 1 \leq i < j \leq s,$ either
\begin{itemize}
\item $\{i,j\}\notin{\mathfrak{B}}$, and $S(i,j)$ reduces to 0 by assumption, or
\item ${\bf T}(i) {\bf T}(j) = {\bf T}(i, j)$  and $S(i,j)$ reduces to 0 by Buchberger's First Criterion, or
\item $S(i,j)$ reduces to 0 by  Buchberger's Second Criterion, since by inductive assumption both $S(i,k)$ and $S(k,j)$ reduce  to 0.
\end{itemize}
\end{proof}

 The following example shows that Corollary \ref{comeBuch} can effectively apply the power granted by  M\"oller Lemma and Buchberger's Corollary~\ref{GMCriCor} only if  the construction outlined in Lemma~\ref{sottoconidisgiunti} is performed on the elements of $M$ after having preliminarily ordered them so that
(\ref{EqTeo}) holds.

\begin{example}\label{Ultimo}
Let $\mathcal J=(M=\{ xy, xz,yz^2\}, \lambda, \tau)$ be the RS in $\TT=\TT[x,y,z]$ with  disjoint cones given by $\tau_{xy}=\TT[x,y]$, $\tau_{xz}=\TT[x,z]\cup \TT[x,y]$, $\tau_{yz^2}=\TT[x,y,z]$ (and tails defined in any way such that $\mathcal J$ be noetherian). In order to decide whether  a marked set $\mathcal F=\{ f_{xy} ,f_{xz} ,f_{yz^2} \}$ on $\mathcal J$  is a basis according with Corollary \ref{comeBuch}  we should    check the reductions of the three S-polynomials $S( f_{xz} ,f_{xy} )= zf_{xy}-yf_{xz}$, $S(f_{yz^2},f_{xy}) =   z^2f_{xy}-  xf_{yz^2}$, $S(f_{yz^2}, f_{xz})=yzf_{xz}- xf_{yz^2}$. However, by Proposition \ref{comeBuchMigli2} it is sufficient to check  the first and  either the   second   or the third pair, as both $xy$ and $xz$ divide $\lcm (xy,yz^2) = \lcm (xz,yz^2) = xyz^2$.

 Note that we have
\begin{eqnarray*}
S(f_{yz^2},f_{xy})-zS( f_{xz} ,f_{xy} )+S(f_{yz^2},f_{xy})
&=& \\
\left( z^2f_{xy}- {\bf  xf_{yz^2}}\right)-z\left(zf_{xy}-{\bf yf_{xz}}\right)+\left(yzf_{xz}- xf_{yz^2}\right) 
&=& 0
\end{eqnarray*}
where $xf_{yz^2}, yzf_{xz}\notin\langle\tau'\mathcal F\rangle$ while $xf_{yz^2}, yf_{xz}\in\langle\tau'\mathcal F\rangle;$ as a consequence we have 
$$g:= yf_{xz}\in\langle\tau'\mathcal F\rangle \not\Longrightarrow zg= yzf_{xz}\in\langle\tau'\mathcal F\rangle$$ 

We further remark that the ordering of the elements of $M$ which follows the construction proposed by Janet \cite{Jan1} has the negative aspect that the first element $yz^2$ to which, according the Staggered Basis construction outlined in Lemma~\ref{sottoconidisgiunti}, we associate $\tau_{yz^2}=\TT$ is of higher degree then the other two elements.

This is the reason why we fail here to obtain the full effect of M\"oller Lemma.
 \end{example}

 It is well-known that the need of storing and ordering all pairs  $\{i,j\}, 1\leq i < j \leq s,$
in order to extract $\mathfrak{B}$ produces a bottleneck and is the weakness of Buchberger's Corollary~\ref{GMCriCor}; all efficient implementation of Buchberger Criteria have the ability of storing only ``useful'' pairs; our approach based on Corollary \ref{comeBuch} shares then same property.
 
\section{Stably ordered reduction structures}\label{sez:ordered}

Another case in which the control proving whether a marked set is a marked basis can be performed via a finite number of  predetermined reductions is the case of 
 {\it  stably ordered RSs} that now we introduce.
 
 In the following Section \ref{casispeciali}, we will examine some significant examples that are included in this case, such as border bases and Pommaret bases; we will see that for each of them we can consider term-ordering free versions.

\begin{definition}\label{fortementeordinatater} Let $\mathcal J=(M,\lambda, \tau)$ be a RS.    We will say that  $\mathcal J$ is { \it  stably ordered}   
by  a well-founded  order  $\Prec$ if 
  taken $x^\alpha,  x^{\alpha'}\in M$   and $x^\eta, x^{\eta'}, x^\epsilon \in \mathcal T$:
 \begin{enumerate}[StOr1:\ ]
\item \label{fortementeordinatater_1} $x^\eta\Succ 1$ for each term $x^\eta \neq 1$
\item \label{fortementeordinatater_2} $x^{\eta'}  \Succ x^{\eta}$ iff  $x^{\eta'+\epsilon} \Succ  x^{\eta+\epsilon}$
\item  \label{fortementeordinatater_3} if  $ x^{\eta+\alpha}=x^{\eta'+\alpha'}\in \cone({x^{\alpha'}})$   and \ $x^\alpha\neq x^{\alpha'}$,    then      $x^{\eta} \Succ x^{\eta'}$
\item \label{fortementeordinatater_4}   if  $   x^{\gamma} \in \lambda_\alpha$, $x^\eta \in \tau_\alpha$   and  $x^{\eta+\gamma}=x^{\alpha'+\eta'} \in \cone(x^{\alpha'})$  then  $x^\eta \Succ x^{\eta'}$.
\end{enumerate} 
\end{definition}

\begin{lemma} \label{exremark} A stably ordered RS  $\mathcal J$ has reduced tails and  disjoint cones.
\end{lemma}
\begin{proof}
Let $\mathcal J=(M,\lambda, \tau)$ be  stably ordered by the well founded  order $\Prec$.  If for some $x^\alpha, x^{\alpha'} \in M$ there is $x^\gamma \in   \lambda_\alpha \cap \cone(x^{\alpha'})$, then  by StOrd\ref{fortementeordinatater_4} (with $x^\eta=1$)     we get $1\Succ x^{\gamma-\alpha'}$ in  contradiction with StOrd\ref{fortementeordinatater_1}. Hence $\mathcal J$ has reduced tails.

\medskip

If there is a term  $ x^{\delta} \in \cone(x^\alpha)\cap \cone(x^{\alpha'}) $ with  $x^\alpha \neq x^{\alpha'}$,  by StOrd\ref{fortementeordinatater_3} we would get the contradiction $x^{\delta-\alpha} \Succ x^{\delta- \alpha'}$ and also $x^{\delta- \alpha'}\Succ x^{\delta- \alpha}$. Therefore  $\mathcal J$ has  disjoint cones.
\end{proof}

 Due to the previous lemma it makes sense the following
 
 \begin{definition}\label{varphi}
 Let $\mathcal J=(M,\lambda, \tau)$ be a RS  stably ordered by  $\Prec$ and  $\varphi \colon J \rightarrow \mathcal T$ be  the function  given by  
 $ \varphi(x^\beta):=x^{\beta-\alpha} $ where $x^\alpha $ is the unique term in $M$ such that $x^\beta \in \cone(x^\alpha)$.
 We will denote by $\pmb{<_\varphi}$ the following relation in $\mathcal T$
$$  x^\beta \pmb{>_\varphi}  x^\delta \hbox{ iff  either } x^\beta \in J,\ x^\delta \in \cN(J) \hbox{ or } x^\beta,   x^\delta \in J  \hbox{ and }  \varphi(x^\beta) \Succ \varphi(x^\delta).$$
\end{definition}

\begin{proposition}\label{isordered}
If  $\mathcal J$ is  stably ordered by the well founded order $\Prec$, then it is noetherian,  ordered by $\pmb{<_\varphi}$, and confluent.
\end{proposition}
\begin{proof}
Clearly, $\pmb{<_\varphi}$ is a well-founded order in $\mathcal T$, since $\Prec$ is.

Moreover, for every $x^\alpha \in M$,  $   x^{\gamma} \in \lambda_\alpha$, $x^\eta \in \tau_\alpha$   we have either $x^{\eta+\gamma}\in \cN(J)$ or $x^{\eta+\gamma}=x^{\alpha'+\eta'} \in \cone(x^{\alpha'})$; in both cases  $x^{\alpha+\gamma} \pmb{>_\varphi}  x^{\eta+\gamma}$: by definition of of $\pmb{<_\varphi} $ in the first case, by  StOrd\ref{fortementeordinatater_4}  in the second one.

 The noetherianity of $\mathcal J$ follows from the fact that it is ordered (Theorem \ref{implicazioni}) and the confluence by  the fact that it has disjoint cones (Lemma \ref{exremark} and Remark \ref{completezza}).
\end{proof}

Let $\mathcal J$ be stably ordered RS and let   $x^{\overline \alpha}\in M $ be  such that $\tau_{\overline \alpha}=\mathcal T$:  such an  element always exists (Lemma \ref{primolemma}), and is unique since the cones are disjoint. We can reformulate the conditions  StOrd1-StOrd4 in terms of  
$ \pmb{<_\varphi} $   has follows:  for every $x^\alpha \in M$

\medskip

\begin{enumerate}
\item [StOr1':\ ]   $x^{\eta + {\overline \alpha}} \pmb{>_\varphi}   x^{{\overline \alpha}}$ for each term $x^\eta \neq 1$;
\item  [StOr2':\ ]  $x^{\eta'+{\overline \alpha}}  \pmb{>_\varphi}    x^{\eta+{\overline \alpha}}$ iff  $x^{\eta'+\epsilon+{\overline \alpha}}  \pmb{>_\varphi}   x^{\eta+\epsilon+{\overline \alpha}}$ for each term $x^\epsilon$;
\item[StOr3':\ ]   if   $x^\eta \notin \tau_\alpha$,    then      $x^{\eta + {\overline \alpha}} \pmb{>_\varphi}  x^{\eta+\alpha}$;
\item[StOr4':\ ]    if  $x^\eta \in \tau_\alpha$ and  $   x^{\gamma} \in \lambda_\alpha$,    then  $x^{\eta +{\overline \alpha}} \pmb{>_\varphi}  x^{\eta + \gamma}$  and $x^{\eta +{\alpha}} \pmb{>_\varphi}  x^{\eta + \gamma}$.
\end{enumerate} 

\medskip

\begin{lemma} \label{utilepermembership}  Let $\mathcal F$ be a marked set over a stably ordered RS $\mathcal J$ and let $\pmb{<_\varphi} $ and  $x^{\overline \alpha}$ be as above.  

Let  us consider   terms $x^\alpha, x^{\alpha'} \in M$,  $x^\eta, x^{\eta'}, x^\delta \in \mathcal T$ and polynomials $g,l\in \mathcal P$.  
\begin{enumerate}[i)]
\item If  $x^\delta$ appears in the support of   $x^\eta f_\alpha -x^{\eta+\alpha}$, then 
$ x^{\eta+\overline \alpha}  \pmb{>_\varphi} x^\delta.$ 
\item If  $g \rightarrow_{\mathcal F \mathcal J}^+ l\downarrow$ and  $x^{\eta+\overline \alpha}  \pmb{>_\varphi} x^\gamma $ for every term  $x^\gamma$ that appears in the support of   $g$,   then $g-l$ has a $x^{\eta+\overline \alpha}-\SLR $  (w.r.t. $ \pmb{<_\varphi} $).
\item In particular, if $x^\eta \notin \tau_\alpha$, then 
  $$x^\eta f_\alpha \rightarrow_{\mathcal F \mathcal J}^+ 0\downarrow \Longleftrightarrow x^\eta f_\alpha   \hbox{ has a  }  x^{\eta+{\overline \alpha}}-\SLR.$$
\end{enumerate} 
\end{lemma}

\begin{proof}  $ i)$   By     Proposition \ref{isordered} the marking of $\mathcal F$ is coherent with the well-founded ordering $\pmb{>_\varphi} $, so that   $ x^{\eta+  \alpha}  \pmb{>_\varphi} x^\delta$. 

        If $x^\eta \in \tau_\alpha$, then $\varphi(  x^{\eta+\alpha}  )=\varphi(x^{\eta+\overline \alpha})$, hence   $ x^{\eta+ \overline  \alpha}  \pmb{>_\varphi} x^\delta$.   
        
          If $x^\eta \notin \tau_\alpha$,   by StOr3'  we have   $x^{\eta + {\overline \alpha}} \pmb{>_\varphi}  x^{\eta + { \alpha}} \pmb{>_\varphi} x^\delta$.
          \medskip
            
              \noindent Item  $ii)$  follows from Proposition \ref{ScritturaDistinti1} \ref{ScritturaDistinti1_iii};  item $iii)$ is a consequence of the previous ones and of StOr3'.
\end{proof}

\begin{theorem}\label{conidisgiunti}  Let $\mathcal F$ be a marked set over a stably ordered  RS  $\mathcal J$.  Then, the property for $\mathcal F$  of being  a marked basis is equivalent to 
 \begin{equation} \label{conidisgiunti_1}     \forall x^{\beta}  \in M,\ \forall x^\varepsilon  \hbox{ minimal in }  \TT \setminus \tau_\beta \hbox{ w.r.t. the divisibility, it holds } x^\varepsilon f_\beta   \rightarrow_{\mathcal F \mathcal J}^+ 0\downarrow .
\end{equation}
 If  moreover  $\mathcal J$ has   multiplicative variables, then it is  also equivalent to the previous ones:
 \begin{equation} \label{conidisgiunti_2} \forall x^\beta  \in M , \ \forall   x_i\notin \tau_\beta \hbox{ it holds }
 x_i f_\beta   \rightarrow_{\mathcal F \mathcal J}^+ 0\downarrow.
\end{equation}
\end{theorem}
\begin{proof} Let $\pmb{<_\varphi} $ and  $x^{\overline \alpha}$ be as above.  
Due to Theorem \ref{sommadiretta} it is clear  that for a marked basis,   \eqref{conidisgiunti_1} and \eqref{conidisgiunti_2} hold. So we only prove the non-obvious implications.

Suppose that   \eqref{conidisgiunti_1}  holds, but $(\mathcal F)$ is not contained in  $\langle \tau \mathcal F  \rangle$. Then, the following set  is nonempty 
$$ U:=\{x^{\eta+\overline \alpha} \ \vert \    \exists  \  f_\alpha  \in   \mathcal F \ \hbox{  s.t. }   \ x^\eta f_\alpha  \notin  \langle\tau \mathcal F\rangle    \ \}.$$
Since $\pmb{<_\varphi} $ is  a  well-founded order on $\mathcal T$,  the set $U$  
has at least a minimal element: suppose that such a minimal element is  $x^{{\gamma+{\overline \alpha}}}$ and that $x^{{\gamma}}f_{ \beta}\notin \langle \tau\mathcal F \rangle $. 
This is possible only if  $x^{{\gamma}}\notin \tau_\beta$ and,  by the assumption \eqref{conidisgiunti_1}, $x^{\gamma}$ is not minimal in   $ \TT \setminus \tau_\beta$ w.r.t. the divisbility.

Let $x^\varepsilon$ be  a  divisor of $x^{{\gamma}}$,  minimal in $\TT \setminus  \tau_{{\beta}}$. By hypothesis $ x^\varepsilon f_\beta    \rightarrow_{\mathcal F \mathcal J}^+ 0\downarrow$, hence it has a    $x^{\varepsilon+\overline \alpha}-\SLR$ (Lemma \ref{utilepermembership}).  
Multiplying by  $x^{\gamma-\varepsilon}$ every polynomial   $x^{\eta_i}f_{\alpha_i}$ of this representation, we obtain   $x^\gamma f_\beta$ as a sum of polynomials  $x^{\gamma-\varepsilon+\eta_i}f_{\alpha_i}$   such that 
$$ x^{\gamma-\varepsilon+\eta_i +\overline \alpha }  \pmb{<_\varphi} x^{\gamma-\varepsilon}\cdot x^{\varepsilon+\overline \alpha}=x^{\gamma+\overline \alpha}. $$
By the minimality of $x^{\gamma+\overline \alpha}$ in $U$, we deduce that $x^{\gamma-\varepsilon+\eta_i}f_{\alpha_i}\in   \langle\tau \mathcal F\rangle  $, hence the contradiction   $x^\gamma f_\beta \in   \langle\tau \mathcal F\rangle $.

\smallskip

The second statement directly follows from the first; in fact
$\{x_i\notin \tau_\beta\}$ is a minimal basis of $\TT \setminus \tau_\beta$.
\end{proof}

\begin{remark}
Consider a polynomial $x^\varepsilon f_\beta$   as stated in  Theorem \ref{conidisgiunti} and suppose that  $x^{\varepsilon+\beta}\in \cone(x^\alpha)$. 
Then
$S(f_\alpha, f_\beta)$  coincides with  $x^\varepsilon f_\beta-x^{\varepsilon + \beta-\alpha}  f_\alpha$. Indeed by minimality of $x^\varepsilon$ in $\TT\setminus \tau_\beta$ each proper divisor  $x^\delta$ of $x^{\varepsilon}$ belongs to $\tau_\beta$ so it cannot also belong to 
 $\cone(x^\alpha)$.
 
Anyway, the condition concerning the S-polynomials is not sufficient to ensure the minimality of  $x^\gamma$ in  $\TT\setminus \tau_\beta$.
In other words, the conditions required in Theorem \ref{conidisgiunti} are weaker than the ones of Corollary \ref{comeBuch}.

\end{remark}

\section{Specializations}\label{casispeciali}  

 Buchberger reduction, mainly after Reeves-Sturmfels theorem, is associated to the idea of coherence with a term order, i.e. the fact that the head terms are bigger than any term in the tails w.r.t. a fixed term order.

What, instead, is wrong, is to associate Gr\"obner bases to a Buchberger reduction viewed as, in our language, a RS with maximal cones.
In fact all representations (and implementations) of Buchberger reductions assume that the available basis $G$ is given as an ordered set of polynomials and that in each step of reduction the reducible term $t$ is systematically reduced with the first element $g\in G$ whose leading term divides $t$.

In this paper, the thing we are more interested in, is the reduction procedure $\rightarrow_{\mathcal F \mathcal J}^+$, associated to a marked set  $\mathcal F$, rather than the marked set (or basis) itself.
The reduction depends both on  $\mathcal F$ and on the RS  $\mathcal J$, and in particular by the set of multiplicative terms.

Considering a monomial ideal $J$ and a set of generators $M$  we can define the RS  $\mathcal J=(M,\lambda, \tau)$ setting 
\begin{itemize}
\item  $M=\{x^{\alpha_1},\ldots,x^{\alpha_s}\}$ an ordered set of generators of a monomial ideal $J$,
\item $\lambda_{\alpha_i}=\{ x^\gamma \in \TT \  \text{ s.t.  }  \ x^\gamma \prec x^{\alpha_i}\}$,
\item  $\cone(x^{\alpha_i}):=\TT\setminus\bigcup_{j=1}^{i-1}\cone(x^{\alpha_j})$.
\end{itemize}
and we obtain the RSs coherent with the  term order $\prec$ and so also noetherian (with  the term order $\prec$ as well funded ordering). In this context,  Reeves-Sturmfels Theorem (Theorem \ref{REEVESTURMFELS}) says that a RS  is  noetherian iff it is coherent with a term order.
Thus Gr\"obner bases relative to  $ \prec$ with initial ideal  $J$ are all and only the marked bases over the RS
 $\mathcal J$. 

If we alternatively set
\begin{itemize}
\item  $M=\{x^{\alpha_1},\ldots,x^{\alpha_s}\}$ an ordered set of minimal monomial basis of a monomial ideal $J$,
\item  $\lambda_{\alpha_i}=\{ x^\gamma \in \TT \  \text{ s.t.  }  \ x^\gamma \prec x^{\alpha_i}\}  \cap \cN(J)$,
\item  $\cone(x^{\alpha_i}):=\TT\setminus\bigcup_{j=1}^{i-1}\cone(x^{\alpha_j})$
\end{itemize}
we get all and the only \emph{reduced} Gr\"obner bases.

A marked set  $\mathcal F$ is a basis iff $\rightarrow_{\mathcal F \mathcal J}^+ $ is confluent. Following Buchberger's algorithm, the test can be performed via the reduction of a limited number of S-polynomials among elements of $\mathcal{ F}$.
\\
Indeed, the Main Theorem of Gr\"obner bases Theory \cite[2.2]{Bruno} declares that a generating set $\mathcal{ F}$ is a Gr\"obner basis if and only if each 
S-polynomial between two elements of $\mathcal{ F}$ reduces to 0; Gebauer-M\"oller criteria \cite{GM2,M}                                                                      allow to reduce the number of S-polynomials to be considered.  

The importance of Buchberger Theory as a tool for solving ideal theoretical problems, gave recently interest to alternative tools for producing Gr\"obner bases; of course the milestones of normal forms given as linear combination of elements in the {\em order ideal} $\cN(J)$ and obtained via the (noetherian) Buchberger reduction are preserved and, after all, were already available to the researchers inspired by Hilbert\footnote{Buchberger reduction can be even read already in \cite{Gor1}.}.

The main contribution by Janet is the introduction of the decomposition of the monomial ideal $J$ into cones of multiplicative sets generated by multiplicative variables.

\begin{definition}  (Janet, 1920) \cite[pp .75-9]{Jan1}   Given a generating set  $M$ of a monomial ideal $J$ and one of its elements $x^\alpha$, a variable $x_j$ is called {\it multiplicative} for $x^\alpha$ w.r.t. $M$  if in $M$ there are no elements $x^\beta$ s.t. $\deg_i(x^\alpha)=\deg_i(x^\beta)$ for each $i>j$ and $\deg_j(x^\alpha)<\deg_j(x^\beta)$.

The \em{class} of $x^\alpha\in M$ is the set $\{x^\beta x^\gamma, x^\gamma\in\TT[\mu_\alpha]\}$ where $\mu_\alpha$ is the set of the multiplicative variables for $x^\alpha$.

Moreover, $M$  is  {\it  complete} if  for each term $x^\gamma\in M$ and each non-multiplicative variable $x_j$ there is a monomial in $M$ whose class contains $x^\gamma x_j$.
\end{definition}
Janet bases are the marked bases over RSs of the following form
\begin{itemize}
\item  $M$ a complete generating set of the monomial ideal $J$,
\item  $\lambda_{\alpha}=\{ x^\gamma \in \TT \  \text{ s.t.  }  \ x^\gamma \prec x^{\alpha}\}  \cap \cN(J)$,
\item  $\tau_{\alpha}:=\TT[\mu_\alpha]$.
\end{itemize}
These are RSs, which are coherent with a term order, have multiplicative variables and disjoint cones.
The RSs of the form defined by Janet need to be coherent with a term order, in order to satisfy noetherianity.

The most important difference between Janet's decomposition in cones and our definition is to give a general rule for defining the multiplicative variables for each term in $M$ by considering the inner relation among the elements of $M$.

This aspect has been inforced in the formulation of Janet's approach proposed by Gerdt and Blinkov \cite{BG3,BG4}.

\begin{definition}[Gerdt-Blinkov, \cite{GB5}]\label{Dinv}
An {\em involutive division}\index{division!involutive}\index{involutive!division} $L$ or {\em $L$-division}
on ${\mathcal T}$ is a relation $\mid_L$ defined, for each finite set $U\subset{\mathcal T}$, on the set
 $U \times {\mathcal T}$ in such a way that the following holds for each $u,u_1\in U$
and $t, t_1\in{\mathcal T}$
\begin{enumerate}
\renewcommand\theenumi{{\rm (\roman{enumi})}}
\item $u \mid_L t \then u \mid t$;
\item $u \mid_L u$;
\item $u \mid_L ut, u \mid_L ut_1 \iff
u \mid_L utt_1$;
\item $u\mid_Lt, u_1\mid_Lt \then \Either 
u\mid_Lu_1 \Or u_1\mid_Lu$;
\item $u \mid_L u_1, u_1 \mid_L t \then u \mid_L t;$
\item if $V \subseteq U$ and $u\in V$ then 
$u \mid_L t$ w.r.t. $U \then u \mid_L t$ w.r.t. $V$.
\end{enumerate}

If $u \mid_L t = uw$, $u$ is called an {\em involutive divisor}
of $t$, $t$ is called an {\em involutive multiple}
of $u$ and $w$ is said to be {\em multiplicative} for
$u$.

If $u \nmid_L t = uw$, $w$ is said to be {\em non-multiplicative} for
$u$.
\qed\end{definition}

This definition, for each set $U$ and each $u \in U$, partitions the set of variables in two subsets
\begin{itemize}
\item $M_L(U,u)$, containing the variables $x_i$
multiplicative for $u$:
$$x_i\in M_L(U,u)\iff u \mid_L u x_i;$$
\item $ NM_L(U,u)$, containing the variables $x_i$
non-multiplicative for  $u$:
$$x_i\in NM_L(U,u)\iff u \nmid_L u x_i.$$
\end{itemize}

Finally, for each involutive division $L$, each finite set $U\subset{\mathcal T}$ and each
$u\in U$,
we denote by $L(u,U)$ the \emph{multiplicative set } for $u$, i.e. the 
set of all the terms $w\in{\mathcal T}$
which are multiplicative for
$u$:
$$L(u,U) := \{w\in{\mathcal T} : 
u \mid_L  uw \}.$$
Remark that condition (iii) implies that each 
$L(u,U)$ is completely characterized by the partition
$$\{x_1,\ldots,x_n\} =  M_L(U,u)\sqcup NM_L(U,u)$$
since
$$L(u,U) = \{
x_1^{a_1}x_2^{a_2}\cdots x_n^{a_n} :
a_i \neq 0 \then x_i\in  M_L(U,u)\}.$$

With this notation it is easy to realize that the definition of
involutive division can be formulated as follows:
\begin{definition}[Gerdt---Blinkov]\label{InvDef} An {\em involutive division} $L$ or {\em $L$-division}
on ${\mathcal T}$ is the assignement, for each finite set $U\subset{\mathcal T}$
and each term $u\in U$ of a submonoid 
$L(u,U)\subset{\mathcal T}$ such that the following holds for each $u,u_1\in U$
and $t, w\in{\mathcal T}$
\begin{enumerate}
\renewcommand\theenumi{{\rm (\alph{enumi})}}
\item $t\in L(u,U), t_1 \mid t \then t_1\in L(u,U)$,
\item if $uL(u,U) \cap u_1L(u_1,U) \neq \emptyset$ then 
$u \in u_1L(u_1,U) $ or $u_1 \in uL(u,U) $;
\item if 
$u_1 = uw$ for some $w\in L(u,U)$, 
then $L(u_1,U) \subseteq L(u,U)$;
\item if $V \subseteq U$
then
$L(u,U) \subseteq L(u,V)$ for each $u\in V$.
\end{enumerate}
\end{definition}

A part from Janet bases, the more important bases defined in terms of involutive divisions considered today are Pommaret bases and Gerdt and Blinkov \cite{GB1,GB2} Janet-like bases.

An adaptation of the theory of Gerdt and Blinkov has been suggested in \cite{Grafo}:
\begin{definition}\label{RelInvDiv}
Let $U \subset \mathcal{T}$ be a finite set of terms.  We say that a \emph{relative involutive division} $L$ is given on $U$ if, for each $u \in U$ a partition 
 $$\{x_1,...,x_n\} =M_L(u,U) \sqcup NM_L(u,U),$$
is given on the set of variables s.t. denoted 
$$L(u,U):=\{x_1^{a_1}\cdots x_n^{a_n} \, \vert \, a_i \neq 0 \Rightarrow x_i \in M_L(u,U)\}, $$
the following two conditions hold:
\begin{enumerate}
\item $\cT(U)=\bigcup_{u \in U} uL(u,U)$;
\item $\forall u,v \in U$, $uL(u,U) \cap vL(v,U) = \emptyset$.
\end{enumerate}
The set $M_L(u,U)$ is called \emph{(relative) multiplicative variable's set},  $NM_L(u,U)$ is called \emph{(relative) non-multiplicative variable's set}, whereas $L(u,U)$  is the set of \emph{(relative) multiplicative terms}.
Denoting by $C_L(u,U):=uL(u,U)$ the \emph{(relative) cone} of $u \in U$, conditions $1-2$ above may be also rewritten as:
\begin{enumerate}
\item[1'.] $\cT(U)=\bigcup_{u \in U} C_L(u,U)$;
\item[2'.] $\forall u,v \in U$, $C_L(u,U) \cap C_L(v,U) = \emptyset$.
\end{enumerate}
\end{definition}

A relative involutive division $L$ on a finite set  $U \subset \mathcal{T}$ satisfies conditions $(i)-(iii)$ of Definition \ref{Dinv};
as regards conditions (iv)-(v) they trivially hold since their hypothesis can never happen, 
because the relative cones are disjoint by definition; what is more important, condition (vi) does not make sense in this context, due to the 
relativity of this involutive division.

This definition has been introduced, following an intuition by Janet,   
\begin{itemize}
 \item for analizing all the decompositions in cones by means of multiplicative variables  of the specific case $U=\mathcal{T}_D$ as a prelimnary step,
 \item for constructing all monomial ideals ${\sf T}$ of leading terms for each potential ideal whose Castelnuovo-Mumford regularity is $D$  and such that the chosen decomposition on $U$ turns out to give  a suitable decomposition also for them and their related escalier $\cN$. 
 \end{itemize}
These ideals/escalier can be identified using a combinatorial graph with the following property:
\begin{quote}
given an element in the ideal ${\sf T}$ (resp. escalier $\cN$), walking backwards (resp. forward) in the graph, we can identify all the other generators of ${\sf T}$ (resp. elements of 
$\cN\cap \mathcal{T}_D$). 
\end{quote}

This approach (which probably can be adapted in order to require that such decompositions preserve some symmetry) has no relation at all with any reduction; in our language it assumes that $\lambda_\alpha=\emptyset$ for all $x^\alpha \in M$, which as we remarked (Example~\ref{Diana}) is trivially noetherian and confluent.
 
\section{The zero-dimensional case and border bases} \label{zerodim}

We examine in depth the  zero-dimensional case, since it is suitable for many observations.

Let $J$ be a monomial ideal in $A[x_1, \dots, x_n]$ s.t. $\cN(J)$ is a finite set. An important concept in many papers on this case is the one of \emph{border}.

\begin{definition} The \rm{border} of $J$ (or of $\cN(J)$) is the set of terms
$$\cB(J):= x_1\cN(J)\cup \dots \cup x_n\cN(J) \setminus \cN(J).$$
\end{definition}

Clearly $\cB(J)$ contains the monomial basis of $J$, but in general as a proper subset.

We can characterize the elements of the border as follows:

$$x^\eta \in \cB(J) \iff \exists  x_j : \    x^\eta/x_j \in \cN(J).$$
It follows then that the divisors of an element in the border are all contained in  $\cN(J) \cup \cB(J)$
In many constructions of marked bases over the border, one considers a fixed term order and supposes that in each marked polynomial the elements in the tails are smaller than the head w.r.t. such a term order; anyway there also exist some bases, marked on $\cB(J)$ without this constraint (see \cite{KKR} and \cite{ABM}).

Note that the notion of border bases was originally introduced in
\cite{MMM1,MMM2}, but  in a context with no connection with RSs  (actually being a
reduction-free approach for computing canonical forms).

Our construction is not compatible with Mourrain improved formulation of border
bases in  \cite{Mou} under the notion of \emph{connected to 1}; indeed there it is not required the head terms to be a
semigroup ideal nor the  escalier to be an order ideal.

We can give a reformulation of these definitions in our language, defining a RS $\mathcal J $ as follows. 
Let $M=\{x^{\alpha_1}, \dots, x^{\alpha_s}\} $ be a list, formed by the elements of $\cB(J) $, ordered in an arbitrary way.
Then, we associate to each term  $x^\gamma$ in $J$ the last term   $x^{\alpha_i}$ of the  $M$   dividing $x^\gamma$.

\begin{table}[!ht]
\begin{center}
\caption{Border bases}\label{Border}
\begin{tikzpicture}[scale=0.84]
\draw [thick] (-3.8,1) -- (-3.8,-3.0);
\draw [thick] (-1.2,1) -- (-1.2,-3.0);
\draw [thick] (12,1) -- (12,-3.0);
\draw [thick] (-3.8,1) -- (12,1);
\draw [thick] (-3.8,-0.6) -- (12,-0.6);
\draw [thick] (-3.8,-1.7) -- (12,-1.7);
\draw [thick] (-3.8,-3.0) -- (12,-3.0);
\node at (-2.4,0.2) [] {  $   M$};
\node at (4.5,0.2) [] { $\begin{array}{l}      \text { The border basis  } \cB(J)=\{x^{\alpha_1}, \dots, x^{\alpha_s}\}  \end{array}$};
\node at (-2.4,-1.2) [] {  $  \lambda$};
\node at (3.0,-1.2) [] {  $\begin{array}{l}   \lambda_{\alpha_i}= \cN(J) \end{array}$};
\node at (-2.4,-2.4) [] {  $  \tau$};
\node at (5.2,-2.4) [] {  $\begin{array}{l} \tau_{\alpha_i}=\{ x^\eta \in \TT \text{  s.t. } \forall j>i : \  x^{\alpha_j} \not\vert \ x^{\eta+\alpha_i}    \}  \end{array}$};
\end{tikzpicture}
\end{center}
\end{table}
Notice that this is actually a RS and that the cones are disjoint, as proved in 
\begin{lemma}\label{order} Under the previous hypotheses (and w.r.t. the previous notation) 
\begin{enumerate}[i)]
\item  \label{order_i} for each  $x^{\alpha_i} \in M$ the set $\tau_{\alpha_i}=\{x^\eta  \in \TT \text{  s.t. } \forall j>i : \  x^{\alpha_j} \not\vert \ x^{\eta+\alpha_i}    \} $  is an order ideal.
\item \label{order_ii} Setting $\cone(x^{\alpha_i})= x^{\alpha_i} \tau_{\alpha_i}$, it holds $\bigcup_{x^{\alpha_i} \in M} \cone(x^{\alpha_i})=J $.
\end{enumerate}
\end{lemma}
\begin{proof}   
i) Let $x^\eta \in \tau_\alpha$ and $x^{\varepsilon} \vert x^{\eta}$. If some $x^{\alpha '}$ subsequent to $x^\alpha$ in the list divides $x^{\varepsilon+\alpha}$ then it divides also $x^{\eta + \alpha}$ and this contradicts $x^\eta \in  \tau_\alpha$. 

ii) We prove that for each  $ x^\beta\in J$ there is a term  $x^{\alpha_i}\in \cB(J)$ s.t.  $x^{\beta-\alpha_i} \in \tau_{\alpha_i}$. This trivially follows by construction and from the fact that  $\cB(J)$ is a generating set for $J$.
\end{proof}

In \cite{KKR} the authors consider   a reduction process w.r.t. a marked set  $\mathcal F$ over $\cB(J)$  (called {\it border pre-basis}) and a procedure of reduction is defined. Roughly speaking,  a  term  $x^\gamma$ in $J$ is reduced by any  element in $\mathcal F$ whose head $x^\alpha$  has maximum degree  among those  in  $\cB  (J)$ dividing $x^\gamma$.   In order to prove the noetherianity of this reduction process,   a function   $ind_{\cB(J)}\colon J\cap \TT\rightarrow \NN$ called {\it index} is defined,  associating to each term $x^\gamma$ in $J$ the   degree of  $x^{\gamma-\alpha}$. 

We obtain a special case of this procedure considering $\mathcal F$ as a marked set over the RS $\mathcal J=(\cB(J), \lambda, \tau)$ as in Table 5 where   the terms in $\cB(J)$  are ordered in increasing order by degree; however the two procedures do not coincide  since  $\mathcal J$ has disjoint cones, while  a monomial $x^\gamma$ may have more than one divisor of  maximum degree in $\cB(J)$.

Furthermore, if we order  the elements of   $\cB(J)$ in increasing order w.r.t. a term order $\prec$ (not necessarily degree compatible), then $\mathcal J$   turns out to be  stably ordered with well-founded order $\prec$: in fact the conditions of definition \ref{fortementeordinatater}  are obviously consequence of the properties of a term order. Therefore,   Theorem \ref{conidisgiunti} gives us a finite set of reductions to control in order to decide whether  $\mathcal F$ is a marked basis.

 Notice that in \cite{KKR} there are no characterizations of marked bases using the reduction procedure;   the presented one  is based, as for   Mourrain's work, on the commutativity of multiplication matrices.

\medskip

We now show in some examples the consequences of modifying the order of the elements of $M=\cB(J)$.

\begin{example}\label{bordocrescente}
Let $J=(x^3, x^2y^2, y^3)\subset A[x,y]$. Its border can be written (in increasing order by degree) as  $x^3, y^3, x^2y^2,  x^3y ,xy^3 $. 
The multiplicative sets of the corresponding RS  $\mathcal J$  are:
$\tau_{x^3}=\TT[x]$, 
$\tau_{y^3}=\TT[y]$, 
$\tau_{x^2y^2}=\{1\}$,  $\tau_{xy^3}=\TT[x,y]$, and
$\tau_{x^3y}=\{x^ay^b,\, a \geq 0, \ 0\leq b\leq 1\}$.  Notice that $\mathcal J$ is not a RS with multiplicative variables. 

According to the criterion presented in   Theorem \ref{conidisgiunti}  in order to know whether a marked set  $\mathcal F=\{f_{x^3}, f_{y^3}, f_{x^2y^2},  f_{x^3y} , f_{xy^3}\}$  is a marked basis we would control the reduction of the following polynomials
$ yf_{x^3}$, $xf_{y^3}$, $xf_{x^2y^2}$, $yf_{x^2y^2}$,  $y^2f_{x^3y}  $.

\smallskip

Now we reorder the terms in $\cB(J)$ in increasing order w.r.t. {\rm DegLex} (induced by $x\prec y$) getting $x^3, y^3, x^3y,  x^2y^2,  xy^3 $. In this case the multiplicative sets   are 
$\tau'_{x^3}= \tau'_{x^3y}=\tau'_{x^2y^2} =\TT[x]$, $\tau'_{y^3}=\TT[y]$ and $\tau'_{xy^3}=\TT[x,y]$.
We get a stably ordered RS  $\mathcal J'$ with multiplicative variables. Following  Theorem \ref{conidisgiunti}  to decide whether $\mathcal F$ is a marked basis, we only have to check the reduction of
$ yf_{x^3}$, $xf_{y^3}$, $yf_{x^3y}  $,  $yf_{x^2y^2}$,  all of them  of \lq\lq linear type\rq\rq.
\end{example}

Anyway, we cannot generalize what observed in the previous example, since reordering the terms of the border w.r.t. {\rm DegLex} (or a degree compatible term order) we do not obtain necessarily a RS with multiplicative variables.

\begin{example}\label{CasoDegLex}
Ordering the  border of  $J=(x^3,xy,y^2)\subset A[x,y]$,   w.r.t {\rm DegLex} ($x\prec y$)\\
\begin{minipage}{8cm}  we obtain  $ xy,y^2,x^3,x^2y $ with  cones 
$\tau_{xy}=\{1\}$, $\tau_{y^2}=\{x^ly^h,\, l\leq1, \ h\geq 0\}$, $\tau_{x^3}=\TT[x]$, $\tau_{x^2y}=\TT[x,y]$, as in the picture.
The term $x^2$ is one of the minimal elements of $\mathcal T \setminus \tau_{y^2}$  w.r.t. divisibility.
This means that in order to verify that a marked set   $\mathcal F$ is also a marked basis we  have also to reduce $x^2 f_{y^2} $,   which is not of \lq\lq linear type\rq\rq.
\end{minipage}
\hspace{1cm}
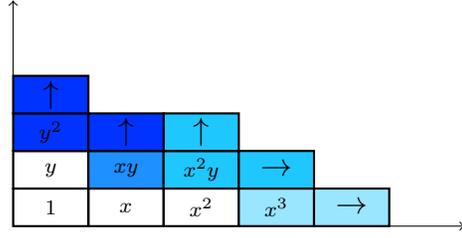
\begin{minipage}{3cm}
\begin{center}
 \begin{tikzpicture}
\draw [thick] (0,0) rectangle (1,.5);
\node at (.5,.25) {$\scriptstyle{1}$};
\draw [thick] (1,0) rectangle (2,.5);
\node at (1.5,.25) {$\scriptstyle{x}$};
\draw [thick] (2,0) rectangle (3,.5);
\node at (2.5,.25) {$\scriptstyle{x^2}$};
\draw [thick, fill=MT1] (3,0) rectangle (4,.5);
\node at (3.5,.25) {$\scriptstyle{x^3}$};
\draw [thick, fill=MT1] (4,0) rectangle (5,.5);
\node at (4.5,.25) {$\rightarrow$};
\draw [thick] (0,.5) rectangle (1,1);
\node at (.5,.75) {$\scriptstyle{y}$};
\draw [thick, fill=MT2] (1,0.5) rectangle (2,1);
\node at (1.5,.75) {$\scriptstyle{xy}$};
\draw [thick, fill=MT3] (2,0.5) rectangle (3,1);
\node at (2.5,.75) {$\scriptstyle{x^2y}$};
\draw [thick, fill=MT3] (3,0.5) rectangle (4,1);
\node at (3.5,.75) {$\rightarrow $};
\draw [thick, fill=MT3] (2,1) rectangle (3,1.5);
\node at (2.5,1.25) {$\uparrow$};
\draw [thick,fill=MT4] (0,1) rectangle (1,1.5);
\node at (.5,1.25) {$\scriptstyle{y^2}$};
\draw [thick,fill=MT4] (1,1) rectangle (2,1.5);
\node at (1.5,1.25) {$\uparrow  $};
\draw [thick,fill=MT4] (0,1.5) rectangle (1,2);
\node at (.5,1.75) {$\uparrow$};
\draw[->] (0,0)--(0,3);
\draw[->] (0,0)--(6,0);.
\end{tikzpicture}
\end{center}
\end{minipage} 
\end{example}

The most convenient choice in general is to forget the degree and reorder the terms w.r.t.  {\rm Lex}.

\begin{theorem}\label{LexDiventaJanet}
Let $J$ be a zero-dimensional monomial ideal and let $M=\cB(J)$ be its border.
 Consider $M$ ordered w.r.t. the lexicographic term order $\prec_ {\rm Lex} $  and let $\mathcal J$ be the associated  RS according to Table \ref{Border}.

 Then $\mathcal J$   has multiplicative variables, which   for every   $x^\alpha \in \cB(J)$ coincide  with the    Janet-multiplicative variables  for  $x^\alpha$ w.r.t. $\cB(J)$, so  
 $\cB(J)$ is a Janet complete system.
\end{theorem}
\begin{proof}  Let $\mu_\alpha$ the set of Janet-multiplicative variables for $x^\alpha \in \cB(J)$. We have to prove that $\tau_\alpha=\TT[\mu_\alpha]$. 

$\supseteq$  Consider $x^\eta \in \mathcal{T}[\mu_\alpha]$ and verify that $x^\eta \in \tau_\alpha$ i.e. that there is no term  $x^\beta \in M$ dividing $x^{\eta+\alpha}$ and s.t. $x^\beta\succ_{Lex}x^\alpha$  .\\
Suppose by contradiction that such a term  $x^\beta $ exists. Let  $x_j$ be s.t.    $\deg_i(x^\beta) =\deg_i(x^\alpha)$ for each $ i>j$ and $\deg_j(x^\beta) >\deg_j(x^\alpha)$. By definition of Janet-multiplicative variable , $x_j \notin \mu_\alpha$.  
We then get a contradiction, since by   $\deg_j(x^{\eta+\alpha})\geq \deg_j(x^\beta)>\deg_j(x^\alpha)$ follows that   $x_j \mid x^\eta$ so, by hypothesis,  $x_j\in \mu_\alpha$.

 \medskip

$\subseteq$\ 
It is sufficient to prove that    $x_j\notin\mu_\alpha$  implies  $x_j  \notin \tau_\alpha$.

If $x_j\notin\mu_\alpha$ , by the definition of Janet-multiplicative variable  there is a term $x^\beta \in \cB(J)$ with $\deg_i(x^\beta) =\deg_i(x^\alpha)$ for each $ i>j$ and $\deg_j(x^\beta) >\deg_j(x^\alpha)$. We prove then that the border contains also an element   $x^{\beta'}$ with $\deg_i(x^{\beta'}) =\deg_i(x^\alpha)$ for each $ i>j$ and $\deg_j(x^{\beta'}) =\deg_j(x^\alpha)+1$, so that $x_j  \notin \tau_\alpha$.

Consider the term $ x^\gamma$ obtained by $x^\alpha$ evaluating at $1$ the variables $x_i$, $i<j$. By construction $x_jx^\gamma\mid x^\beta$ which is in the border; thus $x_jx^\gamma \in \cB(J)\cup \cN(J)$.  Moreover $x_jx^\gamma$ also divides  $x_jx^\alpha$, which belongs to   $ J$. Then, we find the wanted term  $x^{\beta'}\in \cB(J)$  in the set of the multiples of $x_jx^\gamma$ dividing  $x_jx^\alpha$.
\end{proof}

\begin{example}\label{CasoLex} Consider again the monomial ideal of Example \ref{CasoDegLex}.
The border of  $J=(x^3,xy,y^2),$ ordered w.r.t. {\rm Lex} is  $x^3, xy, x^2y, y^2$.  The multiplicative sets of the corresponding RS $\mathcal J''$ are 
$\tau''_{x^3}=\tau''_{x^2y}=\TT[x]$,  
$\tau''_{xy}=\{1\}$,  
$\tau''_{y^2}=\TT[x,y]$.

Thus, $\mathcal J''$ is a stably ordered   RS with multiplicative variables (coinciding with the   Janet-multiplicative ones).

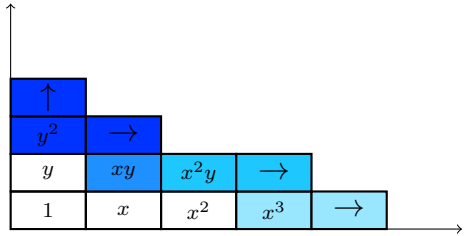
\begin{center}
 \begin{tikzpicture}
\draw [thick] (0,0) rectangle (1,.5);
\node at (.5,.25) {$\scriptstyle{1}$};
\draw [thick] (1,0) rectangle (2,.5);
\node at (1.5,.25) {$\scriptstyle{x}$};
\draw [thick] (2,0) rectangle (3,.5);
\node at (2.5,.25) {$\scriptstyle{x^2}$};
\draw [thick, fill=MT1] (3,0) rectangle (4,.5);
\node at (3.5,.25) {$\scriptstyle{x^3}$};
\draw [thick, fill=MT1] (4,0) rectangle (5,.5);
\node at (4.5,.25) {$\rightarrow$};
\draw [thick] (0,.5) rectangle (1,1);
\node at (.5,.75) {$\scriptstyle{y}$};
\draw [thick, fill=MT2] (1,0.5) rectangle (2,1);
\node at (1.5,.75) {$\scriptstyle{xy}$};
\draw [thick, fill=MT3] (2,0.5) rectangle (3,1);
\node at (2.5,.75) {$\scriptstyle{x^2y}$};
\draw [thick, fill=MT3] (3,0.5) rectangle (4,1);
\node at (3.5,.75) {$\rightarrow$};
\draw [thick,fill=MT4] (0,1) rectangle (1,1.5);
\node at (.5,1.25) {$\scriptstyle{y^2}$};
\draw [thick,fill=MT4] (1,1) rectangle (2,1.5);
\node at (1.5,1.25) {$\rightarrow$};
\draw [thick,fill=MT4] (0,1.5) rectangle (1,2);
\node at (.5,1.75) {$\uparrow$};
\draw[->] (0,0)--(0,3);
\draw[->] (0,0)--(6,0);
\end{tikzpicture}
\end{center}
\end{example}

To conclude, we present a monomial ideal $J$  for  which the border basis (with terms ordered w.r.t. the Lex term order) is not a  Pommaret basis, event though $J$   has both type,  being quasi stable. 

\begin{example}\label{CasononPom}
The terms of the border of $J=(x^3,x^2y,y^3)\subset A[x,y]$  ordered w.r.t. {\rm Lex} are  $x^3, x^2y, x^2y^2, y^3,xy^3$.  We get: 
$\tau_{x^3}=\tau_{x^2y}=\tau_{x^2y^2}=\TT[x]$,  
$\tau_{y^3}=\TT[y]$, 
$\tau_{xy^3}=\TT[x,y]$.

The set of controls one has to perform in order to decide whether a marked set $\mathcal F=\{  f_{x^3}, f_{x^2y},f_{x^2y^2},f_{y^3},f_{xy^3}\}$ involves the reduction of $ yf_{x^3}, yf_{x^2y},yf_{x^2y^2},xf_{y^3}$.

Notice that the sets of multiplicative variables of   $y^3$ and  $xy^3$ do not coincide with the ones w.r.t. Pommaret. Indeed, in the Pommaret  basis  $\{x^3, x^2y, x^2y^2, y^3 \}$ of $J$ there is one term less than in the border basis. At least in this case, in order to determine all the ideals in  $A[x,y]$ whose quotient is a free $A$-module with basis  $\cN(J),$ it would be more convenient to use the Pommaret basis, instead of the border basis. Indeed, the set of controls that are needed involves only  three reductions: $ yf_{x^3}, yf_{x^2y},yf_{x^2y^2}$ 
\end{example}

\section*{Acknowledgements}
The third author thanks Mario Valenzano for his thorough remarks and Felice 
Cardone for useful  
bibliographical suggestions.

\appendix
\section{Functoriality of marked bases}\label{AppendixFunct}

In all this paper we consider marked sets and bases $\mathcal F$ over a RS  $\mathcal J$ with finite tails
  as a set of polynomials in the polynomial ring $\mathcal P_A$ where $A$ is a field.  However, everything holds true if  we  assume that $A$ is any  commutative ring. In fact, the only coefficients in $\mathcal F$   that we need to invert performing a reduction procedure are the leading coefficients.  It is then natural to ask whether our construction is stable under extension of scalars.  In this  appendix we give a positive answer to this question assuming that  $\mathcal J$ has finite tails.

\medskip

There are at least  two functors from the category of commutative rings  to the category of sets that is natural to associate to a RS  $ {\mathcal J }=(M,\lambda, \tau)$ in $\TT$

The functor of marked sets  on $\mathcal J$
\[ 
\MSFunctor{\mathcal J}: \underline{\text{Ring}} \rightarrow \Sets
\]
that associates to any ring  $A$ the set 
$
\MSFunctor{\mathcal J}(A):=\{  \mathcal J-\text {marked sets }  \mathcal F \subset \mathcal P_A  \} $
and to any morphism $\sigma: A \rightarrow B$ the map
\[
\begin{split}
\MSFunctor{\mathcal J}(\sigma):\ \MSFunctor{\mathcal J}(A)\ &\longrightarrow\ \MSFunctor{\mathcal J}(B)\\
\parbox{1.5cm}{\centering $\mathcal F$}\ & \longmapsto\ \sigma(\mathcal F)
\end{split}
\]
where $\sigma(\mathcal F)$ is the set of polynomials that we obtain  from those in $\mathcal F$ replacing each  coefficient $a\in A$ with its image $\sigma(a)$.  More formally, $\sigma(\mathcal F)$ is the image of $\mathcal F$ under the map $\mathcal P_A \rightarrow \mathcal P_B=\mathcal P_A\otimes_\sigma B$

 We observe that this functor is well defined since the coefficient of the distinguished  term $x^\alpha$ in each marked polynomial $f_\alpha\in \mathcal F$  is the unit element and  $\sigma(1_A)=1_B$ for every homomorphism $\sigma\colon A\rightarrow B$.  Hence $\sigma(\mathcal F)$ is indeed a $\mathcal J$-marked in $\mathcal P_B$.
 
 We will denote by $C_\mathcal J$ a set of $N:=\sum_{x^\alpha \in M} \vert \lambda_\alpha \vert$  distinct variables $C_{\alpha,\beta} $ where  $ x^\alpha \in M$ and $ x^\beta \in \lambda_\alpha$. Moreover,  $\mathfrak F$ will denote   the marked set in $\MSFunctor{\mathcal J} (\ZZ[C_\mathcal J ])$ formed by the polynomials $\mathfrak f_\alpha:= x^\alpha + \sum_{x^\beta \in \lambda_\alpha} C_{\alpha,\beta}x^\beta$.

\begin{lemma}
$\MSFunctor{\mathcal J}$  is the functor of points of the ring $\ZZ[C_\mathcal J]$. 
\end{lemma}
\begin{proof}   For every ring $A$ there is a 1--1 correspondence between   $\MSFunctor{\mathcal J}(A)$ and  $Hom (\ZZ[C_\mathcal J] ,A)$. 

In fact we can associate to every homomorphism 
$\pi: \ZZ[C_\mathcal J ] \rightarrow A $ the marked set $ \pi(\mathfrak F) \in \MSFunctor{\mathcal J}(A)$ and, on the other hand, every marked set  $\mathcal F =\{ f_\alpha:= x^\alpha + \sum_{x^\beta \in \lambda_\alpha} c_{\alpha,\beta}x^{\beta}\}\in \MSFunctor{\mathcal J}(A)$ can be obtained in this way considering the homomorphism $\pi_\mathcal F: \ZZ[C_\mathcal J ] \rightarrow A $ given by $\pi_\mathcal F(C_{\alpha, \beta })=c_{\alpha, \beta }$.

Obviously, this  1--1 correspondence commutes with the extension of scalars, since for every homomorphism $\sigma \colon A\rightarrow B$ we have:
$ \sigma(\mathcal F)=\{\sigma( f_\alpha)= x^\alpha + \sum_{x^\beta \in \lambda_\alpha} \sigma(c_{\alpha,\beta})x^{\beta}\}$, and so
$   \pi_{\sigma(\mathcal F)}=\sigma\circ \pi_\mathcal F.$
\end{proof}

As well know,  the category of affine schemes is equivalent to the the category of  rings.  Therefore, we can also define  $\MSFunctor{\mathcal J}$ as a contravariant functor   $\underline{\text{AfScheme}} \rightarrow \Sets$   and say that it is representable  by  the scheme $\mathbb A^N_{\ZZ}=\rm{Spec} (\ZZ[C_\mathcal J  ] )$.

\medskip

 Focusing on the marked bases, we get an even more interesting functor, as a    subfunctor  of $\MSFunctor{\mathcal J}$:

\[\MFFunctor{\mathcal J}(A):=\{ \mathcal J\text {-marked bases }  \mathcal F \subset \mathcal P_A  \}.\]
We now prove that this is in fact a functor.

\begin{lemma}
Let $\mathcal F\in \MFFunctor{\mathcal J}(A)$ and let us consider any morphism $\sigma: A \rightarrow B$. 

Then $\sigma (\mathcal F)$ is a marked basis in $\mathcal P_B$.
\end{lemma}
\begin{proof}  By definition (Definition \ref{BaseMark})  a $\mathcal J$-marked set $\mathcal G\in \mathcal P_R$ is a   basis if and only if $(\mathcal G)_R\oplus \langle \cN(J) \rangle_R=\PP_R$. 

Therefore, by hypothesis we know that  $(\mathcal F)_A\oplus \langle \cN(J) \rangle_A=\PP_A$, and applying $-\otimes_\sigma B$ we get 
 $(\sigma(\mathcal F))_B \oplus \langle \cN(J) \rangle_B=\PP_B$.
\end{proof}

Under the additional assumption  that   $\mathcal J$   is weakly noetherian,  also this subfunctor turns out to be representable by a quotient of $\ZZ[C_\mathcal J]$,  or, equivalently, by an affine subscheme of $\mathbb A^N_{\ZZ}$.  
Similarly to what has been done in  \cite{LR},  we now show how this subscheme can be  obtained.\\
 Let us consider the marked set $\mathfrak F$ in $\mathcal P_{\ZZ[C_\mathcal J]}$   and  compute all  the complete reductions  $ x^{\eta} \mathbf f_{\alpha}  \rightarrow_{\mathbf F \mathcal J}^+ \mathbf g \downarrow$ for every $x^\alpha \in M$ and $x^\eta \in \TT$ and collect  in a set $\mathcal R\subset \ZZ[C_\mathcal J]$    the coefficients of   the monomials   $x^\eta \in \cN((M))$ of all the reduced polynomials  $\mathbf g$.   By Proposition  \ref{caratterizzazioneconfluenza} and  Theorem \ref{sommadiretta}   the marked set $\pi(\mathbf F)$, where $\pi\colon \ZZ[C_\mathcal J]\rightarrow \ZZ[C_\mathcal J]/(\mathcal{R})$, is in fact a marked basis. 

The functor   $\MFFunctor{\mathcal J}$   is represented by the scheme $\MFScheme{\mathcal J}:= \mathrm{Spec}(\ZZ[C_\mathcal J]/(\mathcal{R}))$. 
For a  detailed proof see  \cite{LR}: the arguments presented there  also apply  in our, more general, framework.   

There are many possible applications of the functorial approach to RSs, first of all to the study of Hilbert schemes since the marked schemes $\MFScheme{\mathcal J}$ are flat families. In \cite{BCGR} a subfunctor of $\MFFunctor{\mathcal J}$  for a suitable RS $\mathcal J$ is used to investigate the set of $x_n$-liftings of a given homogeneous ideal. We conclude with an aplication to the theory of marked bases:  for every RS $\mathcal J$ we can check whether the $\mathcal J$-marked sets are bases performing a finite set of reductions.
 
\begin{corollary}  Let $\mathcal J=(M, \lambda, \tau)$  be a weakly noetherian RS with finite tails. 

Then, there exists a finite subset $ G \subset \TT \times M$ such that for every marked set $\mathcal F$  on $\mathcal J$ TFAE:  
\begin{enumerate}[1)]
\item   $\mathcal F$ is a marked basis
\item  for all  $(x^\eta, x^\alpha) \in  G$ and   for all reduction  $  x^{\eta} f_{\alpha}  \rightarrow_{\mathcal F \mathcal J}^+ l\downarrow  $ it holds  $ l=0$.
\item  for all  $(x^\eta, x^\alpha) \in  G$ there is a reduction   $  x^{\eta} f_{\alpha}  \rightarrow_{\mathcal F \mathcal J}^+ 0\downarrow$
\end{enumerate}
\end{corollary}
\begin{proof}
By the noetherianity of the ring $\ZZ[C]$ there exists  a finite set  $\mathcal R' \subset \mathcal R$ that generates the ideal $(\mathcal R)$. For every element  $r\in \mathcal R'$ let us choose  $x^\eta \in \TT$ and $\mathbf f_\alpha \in \mathbf F$ and a reduction   $  x^{\eta}\mathbf  f_{\alpha} \rightarrow_{\mathcal F \mathcal J}^+  \mathbf l \downarrow $  s.t. $r$ is a coefficient in $\mathbf l$; then let us collect in  $G$ the pairs $(x^\eta, x^\alpha)$.  

The thesis is a direct consequence of the fact that $\mathcal F:=\{f_\alpha+\sum_{x^\gamma \in \lambda_\alpha}c_{\alpha \gamma}x^\gamma, \   x^\alpha \in M \}\subset \PP_A$ is a marked basis on $\mathcal J$ if and only if the morphism $\sigma \colon \ZZ[C] \rightarrow A$ given by $\sigma(C_{\alpha \gamma} ) =c_{\alpha \gamma}$ factorizes through $\ZZ[C]/(\mathcal{R})=\ZZ[C]/(\mathcal{R}')$.
\end{proof}

In the case of homogeneous structures,  due to  Gotzmann Persistence \cite{Gotz} and Macaulay Estimate of  Growth \cite[Theorem 3.3]{Gre},  the controls one has to perform can be limited 
to the polynomials whose degree is  bounded from above by   $1+r$, where $r$ is the maximum between  the maximal degree of terms in $M$ and the Castelnuovo-Mumford Regularity of the monomial ideal  $J=(M)$.

A similar upper bound on the degree of polynomials involved in a sufficient set of controls appears also in the affine case in \cite[Theorems 5.1 and 5.4]{BCR}; indeed,  those affine marked  sets   are marked bases if  the following  refinement of the condition \emph{\ref{sommadiretta_ii})} of Theorem \ref{sommadiretta} holds:  $(\mathcal F)_{\leq t}=\langle  {\tau}{\mathcal F}\rangle_{\leq t} $  for some integers  $t \leq  r+1$. 

Finally the recent result proved by \cite{W-W} gives a further bound:
$$(\mathcal F)_{\leq t}=\langle  {\tau}{\mathcal F}\rangle_{\leq t} \mbox{ for all } t\geq2\left(\frac{d^2}{2}+d\right)^{2^{n-1}}+\sum _{j=0}^{n-1} (u d)^{2^j}$$ 
where $d=\max\deg(f : f\in{\mathcal F})$ and $u=\#\mathcal F$.

 In     the above quoted cases  we should  perform  a finite (but in general not small)  number of controls.

\end{document}